\documentclass[11pt]{article} 
\usepackage{graphicx} 
\usepackage{psfrag} 
\usepackage{pstricks}
\usepackage{pst-plot}
\begin{document}

\newcommand{\qed}{\hfill{\setlength{\fboxsep}{0pt}
                  \framebox[7pt]{\rule{0pt}{7pt}}} \newline}
\newcommand{\eqed}{\qquad{\setlength{\fboxsep}{0pt}
                  \framebox[7pt]{\rule{0pt}{7pt}}} }
\newcommand{\st}{\,\colon\,}

\newtheorem{theorem}{Theorem}
\newtheorem{lemma}[theorem]{Lemma}         
\newtheorem{corollary}[theorem]{Corollary}
\newtheorem{proposition}[theorem]{Proposition}
\newtheorem{definition}[theorem]{Definition}
\newtheorem{claim}[theorem]{Claim}
\newtheorem{conjecture}[theorem]{Conjecture}
\newtheorem{remark}[theorem]{Remark}
\newtheorem{question}[theorem]{Question}

\newcommand{\proof }{{\bf Proof.\ }}          


\def\SS{{\mathcal S}}
\def\C{{\mathcal C}}
\def\R{{\mathcal R}}
\def\rdeg{\mathop{\rm rdeg}\nolimits}
\def\Uo{{\overline{U_0}}}
\def\Up{{\overline{U_0'}}}
\def\Us{{\overline{U_0^*}}}
\def\Uh{{\overline{U_0^H}}}
\def\Ui{{\overline{U_0^i}}}


\title{Dominating Sets in Triangulations on Surfaces}
\author{Hong Liu\thanks{Department of Mathematics,
University of Illinois Urbana-Champaign, Urbana, Illinois 61801, USA {\tt hliu36@illinois.edu}.}
\and
Michael J.~Pelsmajer\thanks{Department of Applied Mathematics,
Illinois Institute of Technology, Chicago, Illinois 60616, USA {\tt pelsmajer@iit.edu}.
The author gratefully acknowledges support from NSA Grant H98230-08-1-0043 and the Swiss National Science Foundation, Grant No. 200021-125287/1.}
}
\date{\today}
\maketitle

\begin{abstract}
A {\em dominating set} $D\subseteq V(G)$ of a graph $G$ is a set such that each vertex $v\in V(G)$ is either in the set or adjacent to a vertex in the set.
Matheson and Tarjan (1996) proved that any $n$-vertex plane triangulation 
has a dominating set of size at most $n/3$, and conjectured
a bound of $n/4$ for $n$ sufficiently large.
King and Pelsmajer recently proved this for graphs with maximum degree at most~$6$.
Plummer and Zha (2009) and
Honjo, Kawarabayashi, and Nakamoto (2009) extended the $n/3$ bound to triangulations on surfaces.

We prove two related results:
(i) There is a constant $c_1$ such that any
$n$-vertex plane triangulation with maximum degree at most~$6$ has a dominating set of size at most $n/6+c_1$.
(ii) For any surface $S$, $t\ge 0$, and $\epsilon>0$, there exists
$c_2$ such that for any $n$-vertex triangulation on $S$ with at most $t$ vertices of degree other than~$6$,
there is a dominating set of size at most $n(1/6 +\epsilon)+c_2$.

As part of the proof, we also show that any $n$-vertex triangulation of a non-orientable surface has a non-contractible cycle of
length at most $2\sqrt{n}$.
Albertson and Hutchinson (1986) proved that for $n$-vertex triangulation of an orientable surface other than a sphere has a
non-contractible cycle of length $\sqrt{2n}$, but no similar result was known for non-orientable surfaces.
\end{abstract}


\textbf{Math. Subj. Class.}: 05C69

\section{Introduction}\label{S:introduction}

In this paper, we only consider graphs that are finite, undirected, and simple (no loops or multiple edges) except where specified otherwise.

\medskip

A {\em dominating set} $D\subseteq V(G)$ of a graph $G$ is a set such that each vertex $v\in V(G)$ is either in the set or adjacent to a vertex in the set. The {\em domination number} of $G$, denoted $\gamma(G)$, is defined as the minimum cardinality of a dominating set of $G$.  Domination is very widely studied (see~\cite{mono} for a recent monograph).  A {\em triangulation} on a given surface is a graph embedded on the surface such that every face is bounded by a triangle.

In 1996, Matheson and Tarjan~\cite{mattar} proved $\gamma(G)\leq n/3$ for any $n$-vertex triangulated disc $G$ (which includes all plane triangulations) and that this bound is sharp.
Plummer and Zha~\cite{Plum} recently extended this bound to triangulations on the projective plane and proved
$\gamma(G)\leq \lceil{n/3}\rceil$ for triangulations on the torus or Klein bottle.
Honjo, Kawarabayashi, and Nakamoto~\cite{HKN} obtained $\gamma(G)\leq n/3$ for
triangulations on the torus and the Klein bottle and also for locally planar triangulations
(triangulations of sufficiently high representativity) on every other surface.
Matheson and Tarjan also gave an infinite class of plane triangulations with $n$ vertices that requires $n/4$ vertices to be dominated.
They conjectured that $\gamma(G)\le n/4$ for every plane triangulation $G$ with a finite number of exceptions, such as the octahedron,
which has $6$ vertices and domination number $2$.
\begin{conjecture}[Matheson and Tarjan~\cite{mattar}]\label{MainConjecture}
There exists $n_0$ such that any $n$-vertex plane triangulation with $n>n_0$ has a dominating set of size at most $n/4$.
\end{conjecture}

High degree vertices are helpful when constructing a small dominating set; this motivates the study of Conjecture~\ref{MainConjecture} on graphs with no (or few) high degree vertices.

King and Pelsmajer~\cite{ArxivKingMJP} recently proved Conjecture~\ref{MainConjecture} for plane triangulations with maximum degree at most $6$.
That is, they found $n_0$ such that $\gamma(G)\le n/4$ for any $n$-vertex plane triangulation $G$ with maximum degree at most~$6$
and $n> n_0$.
(The proof has $n_0=4.5 \times 10^6$.)

Our theorems extend the previous results.
First, we show that the degree restriction allows one to prove the following stronger upper bound,
which verifies a conjecture by King and Pelsmajer~\cite{ArxivKingMJP}.
\begin{theorem}\label{T:n/6}
There exists a constant $c$ such that any $n$-vertex plane triangulation with maximum degree at most $6$
has a dominating set of size at most $n/6 +c$.
\end{theorem}

We prove Theorem~\ref{T:n/6} with $c=1.05 \times 10^7$.

\medskip

Maximum degree at most~$6$ implies that there are at most $12$ vertices of degree less than $6$,
and all other vertices have degree~6.
King and Pelsmajer conjectured that their result can be extended to plane triangulations with a bounded number of vertices with degree not equal to~$6$.

\begin{conjecture}[King and Pelsmajer~\cite{ArxivKingMJP}]\label{C2}
For any constant $t$, there exists $n_t$ such that
any $n$-vertex plane triangulation with $n>n_t$ and
at most $t$ vertices of degree other than~$6$
has a dominating set of size at most $n/4$.
\end{conjecture}

They also gave examples~\cite{ArxivKingMJP} showing that the bound $n/4$ in Conjecture~\ref{C2} cannot be improved beyond $n/6+c$; the examples also imply that the bound in Theorem~\ref{T:n/6} is best possible.

One reason to restrict $t$ rather than the maximum degree is to allow us to consider domination of similar degree-restricted triangulations on various surfaces.
Indeed, it follows easily from Euler's Formula that if $S$ is any surface other than the plane (sphere), projective plane, torus, or Klein bottle, then there are no triangulations on $S$ with maximum degree at most~$6$.

\begin{theorem}\label{T:main}
For any surface $S$, integer $t\ge 0$, and $\epsilon>0$, there exists $c=c(S,t,\epsilon)$ such that
any $n$-vertex triangulation on $S$,
with at most $t$ vertices of degree other than $6$, has a dominating set of
size at most $n(\frac{1}{6}+\epsilon)+c$.
\end{theorem}

We prove Theorem~\ref{T:main} with $c=O((g^3+gt^2)/\epsilon)$, where $g$ is the genus of $S$.
Since $n$ is a trivial upper bound for domination number, Theorem~\ref{T:main} is only
meaningful when $c(S,t,\epsilon)<n$, in which case $t=O(\sqrt{n})$.

Since $n(\frac{1}{6}+\epsilon)+c\le \frac{n}{4}$ whenever
$n(\frac{1}{12}-\epsilon)\ge c$, Conjecture~\ref{C2} follows
by setting $\epsilon<\frac{1}{12}$
(for example, let $\epsilon=0.08$) and $n_t=c/(\frac{1}{12}-\epsilon)$.

\medskip

In the proof of Theorem~\ref{T:main}, we need to know that every $n$-vertex triangulation on a surface has a ``small'' non-contractible cycle.
Albertson and Hutchinson~\cite{JoanHutchinson} showed that for any fixed orientable surface,
there is a non-contractible cycle of length at most $\sqrt{2n}$.
We prove a similar bound for non-orientable surfaces.
\begin{theorem}\label{T:non}
Any $n$-vertex triangulation on a non-orientable surface has a non-contractible cycle of length less than or equal to $2\sqrt{n}$.
\end{theorem}

\begin{remark}{\rm
The upper bound of $\sqrt{2n}$ in~\cite{JoanHutchinson} was improved to $O(\sqrt{n/g}\log g)$ in~\cite{Hutch}, where $g$ is the genus of the orientable surface.  However, using this result in our proof
will not reduce $c$ below $c=O((g^3+gt^2)/\epsilon)$ in the orientable case.

We will use the cycles to reduce a surface triangulation to sphere (or plane) triangulations,
as described in Section~\ref{S:def}.
A closely related, well-known method of reducing a graph to a planar graph is to delete a small set of
vertices (a {\em planarizing set}).  Deleting a planarizing set does not necessarily yield a
triangulation, so it is not clear whether it could be used to prove Theorem~\ref{T:main}.  Even if this was possible, it probably would not yield a stronger result, since the best known bound on the size of a minimum planar set is $O(\sqrt{gn})$~\cite{Djidjev,HutchMiller}, which is not much smaller than the total number
of vertices involved in at most $O(g)$ non-contractible cycles that appear during the proof, and
which is the main bottleneck in lowering the value of $c$.
}\end{remark}

By the following construction, the order of magnitude of the bound in Theorem~\ref{T:non}
is best possible:
Take an icosahedron, identify opposite points,
and then triangulate each of the resulting 10 faces with a triangular grid of
$k^2$ triangles.  We get a projective planar triangulation with $10k^2$ faces,
and hence $15k^2$ edges and $n=5k^2+2$ vertices.  Its shortest
non-contractible cycle has length $3k$, which is nearly
$\frac{3}{\sqrt{5}} \sqrt{n}$, or about $1.34164 \sqrt{n}$.

Thus,
the correct bound has order $\sqrt{n}$, and the constant multiple is between
1.34 and 2.

\begin{question}
What is the smallest constant that could replace $2$ in Theorem~\ref{T:non}?
\end{question}

In Section~\ref{S:nonsphere} we explain how to reduce the proof of Theorem~\ref{T:main} to
finding small dominating sets for triangulations of spheres that contain certain pre-specified
subsets of vertices.  We show how to find such a dominating set in Section~\ref{S:sphere},
and prove Theorem~\ref{T:n/6} as well.  The proof of Theorem~\ref{T:n/6}
follow the basic structure of the proof in~\cite{ArxivKingMJP}:
we construct two kinds of dominating sets and show that one must be small enough.
The argument in Section~\ref{S:sphere} that finishes the proof of Theorem~\ref{T:main}
is similar except that one of the sets only dominates part of the graph;
we apply induction to dominate the rest of the graph.

Before that, we review some standard definitions and then introduce some definitions needed for our paper.

\section{Definitions and Preliminaries}\label{S:def}

For any graph $G$, let $V(G)$ and $E(G)$ denote the set of vertices and edges, respectively.
For a vertex $v$ in a graph $G$, $\deg(v)$ and $\deg_G(v)$ each denote the {\em degree} of $v$ in $G$.
$\Delta(G)$ is the maximum degree in $G$.
For any $U\subseteq V(G)$, let $G[U]$ denote the subgraph of $G$ induced by $U$,
which has vertex set $U$ and edge set $\{uv\in E(G) :u,v\in U\}$.

The {\em length} of a path, cycle, or walk $W$ is the number of edges and it is denoted by $|W|$.
Since a walk $W$ may repeat edges, $|W|$ may exceed $|E(W)|$.
If $W$ is a cycle then $|V(W)|=|W|$ and if $W$ is a path then $|V(W)|=|W|+1$.
A {\em chord} of a cycle or walk $W$ is an edge not in $E(W)$ with endpoints in $W$.

The {\em distance} between vertices $u$ and $v$, denoted $d(u,v)$, is the
minimum length of a $u,v$-path.  The distance between two sets (where each set could be a single vertex)
is the minimum length of a path with one endpoint in each set.
For a vertex $v$ in a graph $G$ and any integer $i\ge 0$, let $N_i(v)= \{u\in V(G)\st d(v,u)=i\}$
and let $N_i[v]= \{u\in V(G)\st d(v,u)\le i\}$.  Let $G_i$ be the subgraph of $G$ induced by $N_i[v]$,
that is, $V(G_i)=N_i[v]$ and $E(G_i)$ contains the edges in $G$ that have both endpoints in $N_i[v]$.

A graph is {\em connected} if there
is a $u,v$-path for every pair of vertices $u,v$.  A vertex is a {\em cut-vertex} if its removal
increases the number of components of the graph.  A graph with more than two vertices is
{\em $2$-connected} if it is connected and it has no cut-vertices.
A {\em block} is a maximal subgraph with no cut-vertices.  Each block is either $2$-connected, a
single edge and its endpoints, or an isolated vertex.  For any connected graph $G$, its
{\em block-cutpoint tree} $T$ is defined such that the blocks and cut-vertices of $G$
are the vertices of $T$, and a block $B$ and vertex $v$ of $G$ are adjacent in $T$ if and only if $B$ contains $v$.

For a graph $G$ and vertex set $U\subseteq V(G)$, a {\em Steiner tree} for $U$ is a minimum-size tree that
contains $U$.

\medskip

By a {\em surface}, we generally mean a $2$-manifold without boundary.
For $g\ge 0$, let $S_g$ denote the orientable surface of orientable genus $g$ (a sphere with $g$ handles).
For $g\ge 1$, let $N_g$ denote the non-orientable surface of non-orientable genus $g$ (a sphere with $g$
cross-caps).
By the classification theorem for surfaces, these are all the possible surfaces.
For both $S=S_g$ and $S=N_g$, $g$ is called the {\em genus} of $S$.

Given a graph $G$ on a surface $S$  such that
every face is homeomorphic to a disk, $|V(G)|-|E(G)|+|F(G)|$
is the {\em Euler characteristic} of $S$,
where $F(G)$ is the set of faces.
It is $2-2g$ for $S=S_g$ and $2-g$ for $S=N_g$; thus,
it is independent of the choice of $G$.
For any $G$ we have the {\em degree-sum formula} $\sum_v \deg(v) = 2|E(G)|$.
For triangulations we have $2|E(G)|=3|F(G)|$, so the Euler characteristic equals
$|V(G)|-|E(G)|/3 = \frac{1}{6}\sum_v (6-\deg(v))$.

A closed curve on a surface is either {\em one-sided} or {\em two-sided}, and a surface is {\em orientable}
if and only if it has no one-sided closed curves.
A two-sided closed curve can be {\em contractible}, {\em surface-separating}, or neither; closed curves which
are neither are called {\em essential}.
A {\em plane graph} is a graph drawn in the plane without crossings; a {\em planar graph}
is a graph that can be drawn as a plane graph.
With respect to graph embeddings, the sphere is usually interchangeable with the plane.
See~\cite{MoharThomassen} for further background on graphs on surfaces.

Suppose that $C$ is a non-contractible cycle of $G$ on the surface $S$.
We can obtain a triangulation $G'$ on a related surface $S'$, which we call the \emph{$C$-derived graph of $G$}, as follows.
Cut the surface along $C$, and copy $C$ locally onto each side of the cut so that each vertex
and edge is doubled.
This creates a triangulated surface with one hole if $C$ is
one-sided, and two holes if $C$ is two-sided.
Attach a disk to each hole to create a surface without boundary; this is $S'$.
Add a vertex to each disk with edges to every vertex on the disk boundary; this yields
$G'$, which is a triangulation of $S'$.

If $C$ is one-sided, then $C$ is replaced by a cycle $C_1$ of length $2|C|$, so
$|V(G')|=|V(G)|+|C|+1$, $|E(G')|=|E(G)|+3|C|$, and $|F(G')|=|F(G)|+2|C|$,
where $F(G)$ is the set of faces of the embedded graph $G$.
In this case, the Euler characteristics of $S$ and $S'$ differ by one.
If $C$ is two-sided, then $C$ is replaced by cycles $C_1,C_2$ of length $|C|$, so
$|V(G')|=|V(G)|+|C|+2$, $|E(G')|=|E(G)|+3|C|$, and $|F(G')|=|F(G)|+2|C|$.
In this case, the Euler characteristics of $S$ and $S'$ differ by two.
If $C$ is two-sided, then $C$ may be {\em surface-separating}, in which case
the Euler characteristic of $S'$ equals the sum of the Euler characteristics
of its two components.

Since the Euler characteristic of $S_g$ is $2-2g$
and the Euler characteristic of $N_g$ is $2-g$,
there are only certain possibilities for $S'$,
which we summarize in Table~\ref{t:surfaces}.

\begin{table}[ht]
\begin{center}
\begin{tabular}{c||c|l||c|c}
Case & $S$ & \qquad\quad \ $C$ & $S'$ & \\
\hline
1 & $S_g$ & $2$-sided, non-separating & $S_{g-1}$ & \\
2 & $S_g$ & $2$-sided, surface-separating & $S_{k_1}\cup S_{k_2}$ %
  & $k_1,k_2\ge1$ and $k_1+k_2=g$ \\
3 & $N_g$ & $1$-sided, non-separating  & $N_{g-1}$ & (only if $g\ge 2$) \\
4 & $N_g$ & $1$-sided, non-separating  & $S_{(g-1)/2}$ & (only if $g$ is odd) \\
5 & $N_g$ & $2$-sided, non-separating  & $N_{g-2}$ & \\
6 & $N_g$ & $2$-sided, non-separating  & $S_{(g-2)/2}$ & (only if $g$ is even) \\
7 & $N_g$ & $2$-sided, surface-separating  & $N_{k_1}\cup N_{k_2}$ %
  & $k_1,k_2\ge1$ and $k_1+k_2=g$ \\
8 & $N_g$ & $2$-sided, surface-separating  & $S_{k_1/2}\cup N_{k_2}$ %
  & $k_1/2,k_2\ge1$ and $k_1+k_2=g$ \\
\end{tabular}
\end{center}
\caption{All possible $C$-derived surfaces $S'$, given the non-contractible cycle $C$ on surface $S$ with genus $g\ge1$}
\label{t:surfaces}
\end{table}

The genus of $S_g$ is also called {\em orientable genus} and
the genus of $N_g$ is also called {\em non-orientable genus}.
Although it is somewhat unusual, it will be convenient for us to
say that the {\em non-orientable genus of $S_{g/2}$ is $g$}
(where $g$ is even).  (See Table~\ref{t:surfaces} for motivation.)

An {\em outerplane} graph is a plane graph with all of its vertices incident to its outer face.
For any graph $G$ on a surface $S$, its {\em dual} is a graph on $S$ with a vertex in each face of $G$,
such that whenever two faces share an edge in $G$, the vertices corresponding to those faces are joined
by an edge of the dual.
The {\em weak dual} of an outerplane graph $G$ is the dual minus the vertex in the outer face of $G$;
the weak dual is a forest, and it is a tree if $G$ is $2$-connected.
The outer face of a connected outerplane graph $H$ is bounded by a closed walk; without loss of generality,
we may always assume that the walk is oriented counterclockwise, so that the exterior of $H$ is always to the right.

\medskip

The next two definitions
allow us to state and work with a lemma from~\cite{ArxivKingMJP} (Lemma~\ref{L:1cycleABCDE} in this paper)
in the situation where the maximum degree is not bounded by~$6$.

\begin{definition}{\rm
Consider a closed walk $W=v_0,v_1,\ldots,v_m=v_0$ (indexed by the cyclic group $Z_m$)
that bounds a connected outerplane subgraph $H$ of a plane triangulation $G$, oriented
counterclockwise, so that the exterior of $H$ is always to the right.

For each $i\in Z_m$, let $\rdeg_i(W)$ be
the number of edges incident to $v_i$ from the right---more specifically,
if the edges incident to $v_i$ are ordered such that they are counterclockwise near $v_i$,
then count the ones that are after $v_iv_{i-1}$ and before $v_iv_{i+1}$.
(The edges counted by $\rdeg_i(W)$ lie in the exterior of $H$, and if $v_i$ is not a cut-vertex of $H$,
then $\rdeg_i(W)$ counts every edge incident to $v_i$ which lies in the exterior of $H$.)

The {\em outer degree sequence of $W$} is the cyclic sequence $\rdeg_1(W),\ldots,\rdeg_m(W)$
indexed by $Z_m$.
}\end{definition}

We name a few special types of cyclic sequences that arise in the proof.

\begin{definition}{\rm  The cyclic sequence $d_1,\ldots,d_m$ (indexed by $Z_m$) is
\begin{description}\setlength{\itemsep}{0pt}\setlength{\parsep}{0pt}%
\item[Type A] if $d_i=2$ for all $i\in Z_m$,
\item[Type B] if $d_i=3$ for some $i\in Z_m$ and $d_j=2$ for all $j\in Z_m\setminus\{i\}$,
\item[Type C] if $d_i=4$ for some $i\in Z_m$ and $d_j=2$ for all $j\in Z_m\setminus\{i\}$,
\item[Type D] if $d_i=d_{i+1}=3$ for some $i\in Z_m$ and $d_j=2$ for all $j\in Z_m\setminus\{i,i+1\}$,
\item[Type E] if $d_i=3$ and $d_{j}=1$ for some $i,j\in Z_m$ and $d_k=2$ for all $k\in Z_m\setminus\{i,j\}$.
\end{description}
}\end{definition}

The following definition is from~\cite{ArxivKingMJP}, albeit slightly renamed:

\begin{definition}{\rm
Let $w,\ell,k$ be integers with $w\ge 3$, $\ell\ge 1$, and $0\le k<w$.
A {\em $(w,\ell,k)$-cylinder}, {\em $(w,\ell)$-cylinder}, or {\em triangulated cylinder}
is any plane graph constructed as follows.

Fix an integer $k$ with $0\le k<w$.  Start with
the Cartesian product of a $w$-cycle and a path of length $\ell$.
The vertices can be labeled $z_{a,b}$ with $a$ in the cyclic
group $Z_w$ and $0 \le b\le \ell$.
For each $0\le b< \ell$,
add an edge from $z_{a,b}$ to $z_{a+1,b+1}$ if $0\le a< k$, and
add an edge from $z_{a,b}$ to $z_{a-1,b+1}$ if $k< a\le w$.
All triangles (except those of the form $z_{0,b_0} z_{1,b_1} z_{2,b_2}$ when
$w=3$) are 3-faces.

The parameter $w$ is the {\em width} and $\ell$ is the {\em length}.
The cycles induced by $\{z_{a,0}\st a\in Z_w\}$ and $\{z_{a,\ell}\st a\in Z_w\}$
bound $w$-faces; these cycles are called the {\em boundary cycles}.
The {\em interior} vertices are the vertices not in either boundary cycle.
}\end{definition}

Suppose that $H$ is a $(w,\ell,k)$-cylinder with $0<k<w$ and let $V(H)=\{y_{a,b}\st z\in Z_w, 0\le b\le \ell\}$.
Let $z_{w-a,b}=y_{a,b}$ for each $y_{a,b}\in V(H)$.  With these new vertex names, we see that $H$ is also a $(w,\ell,w-k)$-cylinder.
Thus, we may always assume that $0\le k\le w/2$ for any $(w,\ell,k)$-cylinder.

Note that a $(w,\ell)$-cylinder has exactly $w(\ell+1)$ vertices.

\medskip

We consider one infinite graph: let $G_\infty$ be the infinite $6$-regular triangular grid
(see Figure~\ref{F:infinite-grid} on the left).
There is a pattern of vertices from $G_{\infty}$
that uses every seventh vertex (see the right side of Figure~\ref{F:infinite-grid})
and that dominates $G_{\infty}$;
let $D_\infty\subseteq V(G_{\infty})$ be this (infinite) set of vertices.

\begin{figure}[ht]
\begin{center}
\includegraphics[width=.8\textwidth]{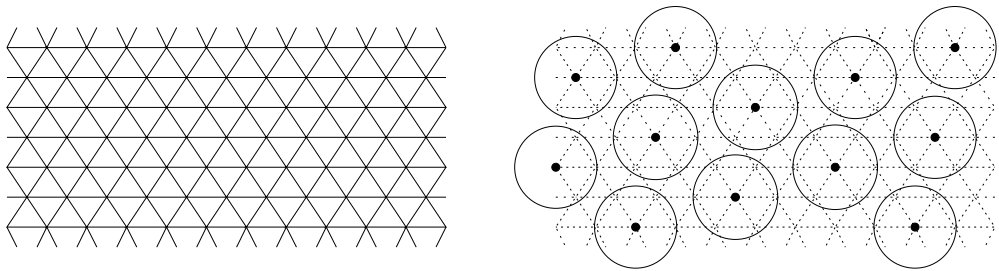}
\end{center}
\caption{$G_\infty$, with a dominating set $D_\infty$ that contains every seventh vertex}
\label{F:infinite-grid}
\end{figure}

\section{Proof of Theorem~\ref{T:main}}\label{S:nonsphere}

Given a triangulation $G$ of an arbitrary surface, we can modify the graph
by repeatedly picking a non-contractible cycle $C$ and replacing its
component $G_C$ with the $C$-derived graph of $G_C$.  This continues until
we have a graph $G'$ that is the disjoint union of triangulations of spheres.
Vertices of degree other than 6 in $G'$ either came from vertices like that
in $G$ or from vertices that at some point were in one of these cycles $C$.
We wish to keep track of these vertices and the way that they cluster.

\begin{definition}\label{D:U}
For any graph $G$, we will always use the notation
$U,U_0,\Uo$ for vertex sets and $d_U$ for an integer that satisfy:
\begin{enumerate}
\item $U$ is the disjoint union of $U_0$ and $\Uo$,
\item $\{v\in V(G): \deg(v)\not=6\} \subseteq U$,
\item every component of $G[U]$ contains at least one vertex of $U_0$, and
\item for each $v\in \Uo$, there is a $u\in U_0$ with $d(v,u)\le d_U$.
\end{enumerate}
(For a graph called $G^*$, we use the notation
$U^*,U_0^*,\Us,d_U^*$ instead, etc.)
\end{definition}

Note that for any triangulation $G$ of a surface, we can satisfy
Definition~\ref{D:U} by letting
$U=U_0=\{v\in V(G): \deg(v)\not=6\}$, $\Uo=\emptyset$, and $d_U=0$.

\medskip

Suppose that $G$ is a triangulation of a surface with $U,U_0,\Uo,d_U$ that satisfy
Definition~\ref{D:U}, and $C$ is a non-contractible cycle in $G$.
The $C$-derived graph $G'$ has two components $G_1,G_2$ if $C$ is surface-separating
and one component if $C$ is non-separating.  We wish to
modify $U,U_0,\Uo,d_U$ so that Definition~\ref{D:U} is satisfied
for the $C$-derived graph and for its components.

\begin{lemma}\label{L:C-derived U}
Suppose that $G$ is a triangulation of a surface with $U,U_0,\Uo,d_U$ that satisfy
Definition~\ref{D:U}, $C$ is a non-contractible cycle in $G$, and $G'$ is the
$C$-derived graph.
If $C$ is 2-sided, let $C_1,C_2$ be the new cycles that
replace $C$ and let $v_1,v_2$ be the new vertices in the disks
bounded by $C_1,C_2$.
If $C$ is 1-sided, then let $C_1$ be the new cycle that
replaces $C$, and let $v_1$ be the new vertex in the disk
bounded by $C_1$.

If $C$ is 2-sided, then Definition~\ref{D:U} is satisfied for $G'$ by
$U'=\left({U\setminus V(C)}\right)\cup V(C_1)\cup V(C_2)\cup\{v_1,v_2\}$,
$U_0'=\left({U_0\setminus V(C)}\right)\cup\{v_1,v_2\}$,
$\Up=\left({\Uo\setminus V(C)}\right)\cup V(C_1)\cup V(C_2)$,
and $d_U'=d_U+1$.  If $G'$ has two components $G_1,G_2$, then
for $i=1,2$, Definition~\ref{D:U} is satisfied for $G_i$
by $U^i=U'\cap V(G_i)$, $U_0^i=U_0'\cap V(G_i)$, $\Ui=\Up\cap V(G_i)$, and $d_U^i=d_U'$.

If $C$ is 1-sided, then Definition~\ref{D:U} is satisfied for $G'$ by
$U'=\left({U\setminus V(C)}\right)\cup V(C_1)\cup\{v_1\}$,
$U_0'=\left({U_0\setminus V(C)}\right)\cup\{v_1\}$,
$\Up=\left({\Uo\setminus V(C)}\right)\cup V(C_1)$,
and let $d_U'=d_U+1$.
\end{lemma}

\begin{proof}
Suppose that $C$ is 2-sided and $U', U_0', \Up, d_U'$ are as stated above.
The properties (1) and (2) of Definition~\ref{D:U} are clearly satisfied for $G'$.
(3) is also satisfied, since any component of $G'[U']$ that intersects
$C_1$ ($C_2$) must also contain $v_1$ ($v_2$).
Since property (4) holds for $G$ and $U,U_0,\Uo,d_U$,
every vertex $v$ of $\Uo\setminus V(C)$ is connected to a vertex of $U_0$
by a path in $G$ with at most $d_U$ edges.
Hence, a minimal path $P$ in $G$ from $v$ to $U_0\cup V(C)$ has at most $d_U$ edges.
$E(P)$ gives us a path in $G'$ from $v$ to $U_0'$ or $C_1\cup C_2$, and in the
latter case it is one more edge away from $v_1$ or $v_2$, which is in $U_0'$.
Since $d_U'=d_U+1$, property (4) holds true for $G'$ and $U',U_0',\Up,d_U'$.

Hence, $U',U_0',\Up,d_U'$ satisfy Definition~\ref{D:U} for $G'$.
If $C$ is surface-separating, then intersecting the sets with
each component $G_1,G_2$ clearly gives sets that satisfy
Definition~\ref{D:U} for each component with $d_U'$.

The same sort of argument works if $C$ is 1-sided.
(In this case, $C$ cannot be surface-separating.)
\qed\end{proof}

\begin{remark}\label{R:D'}{\rm
Still using the terminology of Lemma~\ref{L:C-derived U}:
Note that $|U_0'|-|U_0|\le 2$ if $C$ is 2-sided, and
$|U_0'|-|U_0|\le 1$ if $C$ is 1-sided.
Also, note that $|\Up|-|\Uo|\le 2|C|$ (in both cases).

Moreover, suppose that we have a dominating set $D'$
of $G'$ that contains $U'$.  $D'$ contains $V(C_1)\cup V(C_2)\cup\{v_1,v_2\}$
if $C$ is 2-sided and $D'$ contains $V(C_1)\cup\{v_1\}$ if $C$ is 1-sided.
If we replace these vertices in $D'$ by $V(C)$, we get a vertex set $D$
which is a dominating set for $G$ and contains $U$.
Recall that $G'$ has $|C|+1$ or $|C|+2$ more vertices than $G$, depending on whether
$C$ is 1-sided or 2-sided.  Then $|D'|-|D|=|V(G')|-|V(G)|$.
}\end{remark}

Now consider graphs for which each component is a triangulation of a surface.
If we repeatedly find a non-contractible cycle $C$,
get the $C$-derived graph, and apply Lemma~\ref{L:C-derived U}, we will end up
with a graph $G^*$ with $U^*,U_0^*,\Us,d_U^*$ that
satisfies Definition~\ref{D:U}.

Let $g_0,g_1,g_2$ be the number of surface-separating, (non-separating) 1-sided,
and non-separating 2-sided cycles $C$ used during that process.  Let $\sum|C|$
denote the sum of cycle-sizes, taken over all cycles used during the process.
The total number of $2$-sided cycles used is $g_0+g_2$, so
$|V(G^*)|-|V(G)|=\sum|C|+2g_0+g_1+2g_2$. By Remark~\ref{R:D'},
$|U_0^*|-|U_0|\le 2g_0+g_1+2g_2$ and $|\Us|-|\Uo|\le 2\sum|C|$.

Suppose that $D^*$ is a dominating set of $G^*$ that contains $U^*$.
Then by repeatedly applying Remark~\ref{R:D'},
we get a dominating set $D$ of $G$ that contains $U$, such that $|D^*|-|D|=|V(G^*)|-|V(G)|$.

Now, we are ready to start the proof of Theorem~\ref{T:main}.

Suppose that we are given a triangulation $G$ of an arbitrary surface $S$
with $n$ vertices, at most $t$ of which have degree other than~6.  We will
modify the graph and surface until we have a graph $G^*$ which is the disjoint
union of triangulations of spheres, by repeatedly
picking a minimum non-contractible cycle $C$ in any current component $H$ that
is on a non-spherical surface, and replacing $H$ by its $C$-derived graph $H'$.
Let $g_0,g_1,g_2$ and $\sum|C|$ be as defined above.

Let the {\em genus-sum} of a graph be the sum of the genuses of its components.
Observe that the process ends precisely when the genus-sum reach zero.

According to Table~\ref{t:surfaces}:
When $C$ is surface-separating, the number of components increases by one,
but the genus-sum is unchanged.  When $C$ is non-separating,
the orientable genus-sum decreases by $1$ if $C$ is $2$-sided,
and the non-orientable genus-sum decreases by $j$ if $C$ is $j$-sided.

If $G=S_g$, then the cycle $C$ is 2-sided in each step, so $g_1=0$ and
it takes $g$ steps to reduce the genus-sum to zero, so $g_2=g$.
If $G=N_g$, we must have $g_1+2g_2=g$ in order for the genus-sum to be reduced to zero.
In each step where $C$ is surface-separating, the number of components with
non-zero genus is increased by one, so $g_0\le g-1$, and
in the end we have $g_0+1\le g$ components.  

We may assume that
$|U_0|=t$, $|\Uo|=0$, and $d_U=0$, by the comment following Definition~\ref{D:U}.
Then, we may conclude that
$|\Us|\le 2\sum|C|$,
$|U_0^*|\le t+2g_0+g_1+2g_2$, $|V(G^*)|=n+\sum|C|+2g_0+g_1+2g_2$,
$|D^*|=|D|+\sum|C|+2g_0+g_1+2g_2$, and $d_U^*\le g_0+g_1+g_2\le 2g-1$.

We still need to find a dominating set $D^*$ of $G^*$ that contains $U^*$.
To do this, we find a dominating set for each of its components $H$, which
are triangulations of spheres such that $U^*\cap V(H)$, $U_0^*\cap V(H)$,
$\Us\cap V(H)$, and $d_U^*$ satisfy Definition~\ref{D:U}.

\begin{lemma}\label{L:sphere}
Let $G$ be a triangulation of the sphere with $n$ vertices and
suppose that $U,U_0,\Uo,d_U$ satisfies Definition~\ref{D:U}.
Then $G$ has a dominating set $D$ that contains $U$ such that
\[
|D| \le \frac{n}{6} +
3(|U_0|-1)(2\sqrt{3n} + 2d_U +9) +
\frac{3}{2}|\Uo| +
\frac{1}{3}.
\]
\end{lemma}

We defer the proof of Lemma~\ref{L:sphere} until later;
For now, we assume that it is true.
Applying Lemma~\ref{L:sphere} to each component of $G^*$,
we get dominating sets for each component whose union $D^*$
is a dominating set for $G^*$.  Hence,

\[
|D^*|\le \sum_{H}\left(
\frac{|V(H)|}{6} +
3(|U_0^H|-1)(2\sqrt{3|V(H)|} + 2d_U^* +9) +
\frac{3}{2}|\Uh| +
\frac{1}{3}\right),
\]
where the sum is taken over all components $H$ of $G^*$.
Note that $\sum|V(H)|=|V(G^*)|$, $\sum|U_0^H|=|U_0^*|$, $\sum|\Uh|=|\Us|$,
and $\sum (|U_0^H|-1)(2\sqrt{3|V(H)|})\le
 \sum (|U_0^H|-1)\sum (2\sqrt{3|V(H)|})$.
There are $g_0+1$ components $H$ of $G^*$, so
$\sum\sqrt{|V(H)|} \le
 (g_0+1) \sqrt{(\sum |V(H)|)/(g_0+1)} = \linebreak[2]
 \sqrt{(g_0+1)|V(G^*)|}$, and we get
\[ |D^*|\le
\frac{|V(G^*)|}{6} +
3\left(|U_0^*|-g_0-1\right)
\left( 2\sqrt{3(g_0+1)|V(G^*)|} + 2d_U^* +9 \right) +
\frac{3}{2}|\Us| +
\frac{1}{3}(g_0+1).
\]

Note that $|U_0^*|-g_0-1\le t+g_0+g_1+2g_2-1$.
Also, $|D^*|=|D|+\sum|C|+2g_0+g_1+2g_2$,
$d_U^*\le 2g-1$, $|U_0^*|\le t+2g_0+g_1+2g_2$,
$|\Us|\le2\sum|C|$, and $g_0\le g-1$, so
\[ |D|\le
\frac{|V(G^*)|}{6} +
3\left(t+g+g_1+2g_2-2\right)
\left( 2\sqrt{3g|V(G^*)|} + 4g +7 \right) +
2\sum|C| -g_1-2g_2+\frac{1}{3}.
\]
Also, $|V(G^*)|=n+\sum|C|+2g_0+g_1+2g_2$.
To continue, we need an upper bound for $\sum|C|$.

As stated in Section~\ref{S:introduction}, every $n$-vertex triangulation of a
non-spherical surface has a
non-contractible cycle $C$
with $|C|\le \sqrt{2n}$ if the surface is orientable
and $|C|\le 2\sqrt{n}$ if the surface is non-orientable.

Let $f(m)=f^{(1)}(m)=\lfloor{m+\sqrt{2m}+2}\rfloor$ for $m\ge 3$.
For $i>1$ and $m\ge 3$, let $f^{(i)}(m)=f(f^{(i-1)}(m))$.
Given a triangulation of an orientable surface with at most $m$ vertices,
there is a non-contractible cycle $C$ such that the $C$-derived graph
has at most $f(m)$ vertices.
If $G$ is embedded on $S_g$, then
there are $g_2+g_0$ cycles considered during the process that produces $G^*$,
and at every stage of the process, every component is orientable.
Therefore, $|V(G^*)|\le f^{(g_2+g_0)}(n)$, where $n=|V(G)|$.

The size of each $C$ in the process depends on the size of the component $H$
containing $C$, which could be part of a graph at any stage in the process
before the end (when $G^*$ has been obtained).
The last $C$ considered in the process cannot be surface-separating, so
each time we consider a new cycle $C$,
at most $g_2-1$ non-separating cycles $C$ have already been considered.
Also,
surface-separating cycles produce two components which
each have fewer vertices than their source component.  Therefore,
the number of vertices in a component $H$ that contains any of the cycles $C$
is at most $f^{(g_2-1)}(n)$.  Since there are $g_2+g_0$ cycles $C$ considered
during the entire process, we obtain
$\sum |C| \le (g_2+g_0) \sqrt{2 f^{(g_2-1)}(n)}$.
Since $g_2=g$ and $g_2+g_0\le 2g-1$, we get
$|V(G^*)|\le f^{(2g-1)}(n)$ and $\sum |C| \le (2g-1) \sqrt{2 f^{(g-1)}(n)}$.

Let $F(m)=F^{(1)}(m)=\lfloor{m+2\sqrt{m}+1}\rfloor$ for $m\ge 3$.
For $i>1$ and $m\ge 3$, let $F^{(i)}(m)=F(F^{(i-1)}(m))$.
Note that $F(m)\ge f(m)$ for all $m\ge 3$.
Given a triangulation of a non-orientable surface with at most $m$ vertices,
there is a non-contractible cycle $C$ such that the $C$-derived graph
has at most $F(m)$ vertices.
If $G$ is embedded on $N_g$, then there are $g_1+g_2+g_0$ cycles $C$
considered during the process, and there can be orientable and
non-orientable components during the process.
Since $f(m)\le F(m)$,
we have $|V(G^*)|\le F^{(g_1+g_2+g_0)}(n)$.
As before, taking the $C$-derived graph does not increase the size of any component $H$
which contains some non-contractible cycle from later in the process, if $C$ is
surface-separating cycles or if $C$ is the last non-separating cycle considered.
Continuing as before, in this case we can obtain
$\sum |C| \le (g_1+g_2+g_0) 2 \sqrt{F^{(g_1+g_2-1)}(n)}$.
Since $g_1+g_2+g_0\le 2g-1$ and $g_1+g_2-1\le g-1$, we get
$|V(G^*)|\le F^{(2g-1)}(n)$ and $\sum |C| \le (2g-1) 2 \sqrt{F^{(g-1)}(n)}$.

Using these bounds for $|V(G^*)|$ and $\sum|C|$,
we can rewrite our bound for $|D|$
in terms of $f^{(i)}(n)$ when $G$ is embedded on an orientable surface,
and in terms of $F^{(i)}(n)$ when $G$ is embedded on an orientable surface.

Thus, we need to find bounds on $f^{(i)}(n)$ and $F^{(i)}(n)$.

\smallskip\noindent
{\bf Claim:}
$f^{(i)}(n)\le n+i\sqrt{2n}+i^2+3i-2$ and
$F^{(i)}(n)\le n+2i\sqrt{n}+i^2$.

\noindent
{\it Proof by induction:}  For $i=1$, $f^i(n)=f(n)\le n+\sqrt{2n}+1+3-2$, as desired.

Suppose that $i\ge 2$.
By induction, $f^{(i-1)}(n)\le n+(i-1)\sqrt{2n}+(i-1)^2+3(i-1)-2 =
n+i\sqrt{2n}-\sqrt{2n}+ i^2+i-4$.
Therefore, $f^{(i)}(n)= \lfloor{f^{(i-1)}(n) +\sqrt{2f^{(i-1)}(n)}+2}\rfloor$ is at most
\[
n+i\sqrt{2n}-\sqrt{2n}+ i^2+i-4+\sqrt{2\left[n+i\sqrt{2n}-\sqrt{2n}+ i^2+i-4\right]}+2.
\]
This is at most the claimed upper bound $n+i\sqrt{2n}+i^2+3i-2$ if and only if
\[
\sqrt{2\left[n+i\sqrt{2n}-\sqrt{2n}+ i^2+i-4\right]} \le \sqrt{2n}+2i.
\]
Equivalently (squaring both sides), this is
\[
2\left[n+i\sqrt{2n}-\sqrt{2n}+ i^2+i-4\right] \le 2n +4i\sqrt{2n}+4i^2.
\]
which is true if and only if
\[
0 \le (2i+2)\sqrt{2n}+2i^2-2i+8,
\]
which is true.

To prove $F^{(i)}(n)\le n+2i\sqrt{n}+i^2$,
observe that
$F(n)=\lfloor{(\sqrt{n}+1)^2}\rfloor$ and that
the desired bound can be rewritten as
$F^{(i)}(n)\le (\sqrt{n}+i)^2$.

When $i=1$, we have $F^{(1)}(n)=F(n)\le (\sqrt{n}+1)^2$, as desired.
Suppose that $i\ge 2$.
By induction, $F^{(i-1)}(n)\le (\sqrt{n}+i-1)^2$.
Then $F^{(i)}(n) = (\sqrt{F^{(i-1)}(n)}+1)^2
\le ((\sqrt{n}+i-1)+1)^2 = (\sqrt{n}+i)^2$.
Thus the claim is proved.

\medskip

Note that $\sqrt{f^{(i)}(n)}\le\sqrt{F^{(i)}(n)}\le \sqrt{n} + i$.
Now we are ready to finish the proof of Theorem~\ref{T:main}.

\medskip

First, consider the case that $G$ is a triangulation of $S_g$ ($g\ge 1$).  We have
\[|V(G^*)|\le f^{(2g-1)}(n) \le n + (2g-1)\sqrt{2n}+(2g-1)^2+3(2g-1)-2,\]
\[\sqrt{|V(G^*)|}\le \sqrt{f^{(2g-1)}(n)} \le \sqrt{n}+2g-1,\]
and
\[\sum|C|\le (2g-1)\sqrt{2f^{(g-1)}(n)}\le (2g-1)\sqrt{2}(\sqrt{n}+g-1).\]
Together with
\[ |D|\le
\frac{|V(G^*)|}{6} +
3\left(t+g+g_1+2g_2-2\right)
\left( 2\sqrt{3g|V(G^*)|} + 4g +7 \right) +
2\sum|C| -g_1-2g_2+\frac{1}{3},
\]
$g_0\le g-1$, $g_1=0$, and $g=g_2$,
we can get $|D|\le \frac{n}{6}+a\sqrt{n}+b$, where
\[a=
\frac{\sqrt{2}}{6}(2g-1)+6\sqrt{3g}(t+3g-2)+2\sqrt{2}(2g-1)
\]
and
\[b=
3(t+3g-2)\left(2\sqrt{3g}(2g-1)+4g+7\right)+2\sqrt{2}(2g-1)(g-1)+\frac{2}{3}g^2+\frac{7}{3}g-\frac{1}{3}
.\]
Note that $a$ and $b$ depend only on $g$ and $t$, with $a=O(g^{1/2}(g+t))$ and $b=O(g^{3/2}(g+t))$.

Next, consider the case that $G$ is a triangulation of $N_g$.  We have
\[ |V(G^*)|\le F^{(2g-1)}(n) \le n+ 2(2g-1)\sqrt{n}+ (2g-1)^2, \]
\[\sqrt{|V(G^*)|}\le \sqrt{n}+2g-1,\]
and
\[\sum|C|\le (2g-1)2\sqrt{F^{(g-1)}(n)}\le 2(2g-1)(\sqrt{n}+g-1).\]
Using the same bound for $|D|$ with $g_0\le g-1$ and $g_1+2g_2=g$,
we again can get $|D|\le \frac{n}{6}+a\sqrt{n}+b$, but
this time with
\[a=
\frac{1}{3}(2g-1)+6\sqrt{3g}(t+2g-2)+2(2g-1)
\]
and
\[b=
3(t+2g-2)\left(2\sqrt{3g}(2g-1)+4g+7\right)+4(2g-1)(g-1)+\frac{2}{3}g^2-\frac{2}{3}g+\frac{1}{2}
.\]
Again, $a$ and $b$ depend only on $g$ and $t$, with $a=O(g^{1/2}(g+t))$ and $b=O(g^{3/2}(g+t))$.

If $G$ is a triangulation of the sphere $S_0$, a direct application of Lemma~\ref{L:sphere}
with $|U_0|=t$, $|\Uo|=0$, and $d_U=0$ gives
$|D|\le \frac{n}{6}+a\sqrt{n}+b$ with
$a=6\sqrt{3}(t-1)$ and $b=27(t-1)+\frac{1}{3}$.
So $a=O(g^{1/2}(g+t))$ and $b=O(g^{3/2}(g+t))$ is valid for all surfaces.

To prove Theorem~\ref{T:main}, it suffices to have
$\frac{n}{6}+a\sqrt{n}+b \le
n(\frac{1}{6}+\epsilon)+c$, for some $c=c(S_g,t,\epsilon)$.
Let $c=a^2/(4\epsilon)+b$.
Then the previous inequality is equivalent to
$0\le n\epsilon -a\sqrt{n} +a^2/(4\epsilon)$,
which is true since the right side equals
$\left(
\sqrt{n\epsilon}-\frac{a}{2\sqrt{\epsilon}}
\right)^2$.

Since $a^2=O(g^3)$ and $b=O(g^{5/2})$ for $g\ge t$,
and $a^2=O(gt^2)$ and $b=O(g^{3/2}t)$ for $t\ge g$,
we get $a^2+b=O(a^2)=O(g^3+gt^2)$.
This yields $c=O(a^2/\epsilon)=O((g^3+gt^2)/\epsilon)$.

\begin{remark}{\rm
There must be a constant $c'$ such that
$c=\left(c'(g^3+gt^2)/\epsilon\right)\left(1+o(1)\right)$.
This constant is different for $S_g$ and $N_g$.

Using the stronger Hutchinson bound of $|C|\le O(\sqrt{n/g}\log g)$
instead of $|C|\le \sqrt{2n}$ for the orientable case
would improve $c'$, but not $O((g^3+gt^2)/\epsilon)$.

The argument for the non-orientable case would also work for
$S_{2g}$, where $g$ is non-orientable genus.  But the constant $c'$
we would get is not as good as the current $c'$ for the orientable case.
}\end{remark}

\section{Proof of Lemma~\ref{L:sphere} and Theorem~\ref{T:n/6}}\label{S:sphere}

For convenience, we restate Lemma~\ref{L:sphere} and Theorem~\ref{T:n/6} as a single lemma.

\begin{lemma}\label{L:sphereplus}
Let $G$ be a triangulation of the sphere with $n$ vertices and
suppose that $U,U_0,\Uo,d_U$ satisfies Definition~\ref{D:U}.
Then $G$ has a dominating set $D$ that contains $U$ such that
\[
|D| \le \frac{n}{6} +
3(|U_0|-1)(2\sqrt{3n} + 2d_U +9) +
\frac{3}{2}|\Uo| +
\frac{1}{3}.
\]
Moreover, if $\Delta(G)\le6$ and $U=U_0=\{v\in V(G):\deg(v)\not=6\}$, then
\[
|D| \le \frac{n}{6} +
1.05\times 10^7.
\]
\end{lemma}

By the comment following Definition~\ref{D:U}, when $\Delta(G)\le6$
we may assume that $U=U_0=\{v\in V(G):\deg(v)\not=6\}$ and $d_U=0$;
hence Lemma~\ref{L:sphereplus} does indeed imply Theorem~\ref{T:n/6}.

\medskip

From Euler's formula it follows that
\begin{equation}\label{eq:1}
\sum_{u\in U}(\deg(u)-6)=-12.
\end{equation}
$G$ has minimum degree at least~3 since otherwise $G$ must be a triangle,
in which case any one vertex gives us a sufficiently small dominating set.
Then, by Equation~\ref{eq:1}, $|U|\ge 4$.
Note that if $\Delta(G)\le 6$ and $U=\{v\in V(G):\deg(v)\not=6\}$,
then by Equation~\ref{eq:1}, $|U|\le 12$.

Let $T_0$ be a {\em Steiner tree} for $U_0$ in $G$; that is, let $T_0$ be a tree in $G$
such that $U_0 \subseteq V(T_0)$ and $T_0$ is of minimum size.
Let $U_0' = U_0 \cup \{v\in V(T_0): \deg_{T_0}(v)\not=2\}$.

Suppose that $|U_0|>1$.  Let $L_0$ be the set of leaves in $T_0$; then
$|L_0|\ge 2$ and $L_0\subseteq U_0$.  One can prove by induction
that a tree with $k\ge 2$ leaves has at most $k-2$ vertices of degree
greater than $2$.  Then
we have $|U_0'|\le 2|L_0|-2 \le 2|U_0|-2$.

Since $U_0'$ contains $\{v\in V(T_0): \deg_{T_0}(v)\not=2\}$,
$E(T_0)$ can be partitioned by the maximal paths in $T_0$ with no internal vertices in $U_0'$
(and endpoints in $U_0'$).  There are $|U_0'|-1$ such paths;
let $P_0$ be such a path of maximum length.
Then $|P_0|(|U_0'|-1)\ge |E(T_0)|$.
Since $|V(T_0)| = |E(T_0)|+1$ and $|U_0'| \le 2|U_0|-2$, we get
\[
|V(T_0)| \le (2|U_0|-3)|P_0|+1.
\]
(If $\Delta(G)\le6$ and $U=U_0=\{v\in V(G):\deg(v)\not=6\}$, then $|U_0|\le 12$, so
$|V(T_0)| \le 21|P_0|+1$.)

If $|U_0|\le 1$, then by $|U|\ge 4$ and Definition~\ref{D:U}(3),
we must have $|U_0|=1$.
In this case, $|V(T_0)|=1$.

Since every component of $G[U]$ contains a vertex of $U_0$ and
$T_0$ contains $U_0$, $T_0\cup G[U]$ is connected.
Let $T$ be a spanning tree in $T_0\cup G[U]$.
So, $T$ contains $U$ and $|V(T)|\le |V(T_0)|+|\Uo|$.
(If $|U_0|\le 1$, $|V(T)|=1+|\Uo|$.)

Next, we define $G'$: make two copies of each edge of $T$ and, for each vertex $v\in V(T)$, make
$\deg_T(v)$ copies of $v$.  Draw these all near the original edges and vertices,
and create incidences in the natural way so that we obtain a plane
graph with one face $f_T$ that contains $T$ (before $T$ is deleted), and the other
faces are all $3$-faces (that correspond to the faces of $G$). (See Figure~\ref{F:construct G'} for an example.)
Note the boundary of $f_T$ is a cycle.
For convenience, let us reembed
$G'$ in the plane such that $f_T$ is the outer face.
Ignoring $f_T$, we have a triangulated disc.
Let $V_T'$ be the vertices in $G'$ copied from $V(T)$; then $G'-V_T'=G-V(T)$.

\begin{figure}[ht]
\begin{center}
\includegraphics{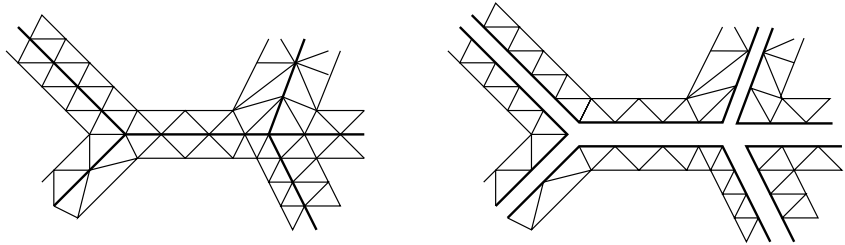}
\end{center}
\caption{An example of constructing $G'$ near a portion of $T$}
\label{F:construct G'}
\end{figure}

The following lemma was originally stated only for the case where $G'$ is constructed from a graph $G$ that has maximum degree at most~6,
but it applies (with the exact same proof) whenever $G'$ is a triangulated disk and all its interior vertices
have degree equal to~$6$.

\begin{lemma}\label{C:map}[King and Pelsmajer~\cite{ArxivKingMJP}]
Suppose that $G'$ is a triangulated disk and all its interior vertices have degree equal to~6.
Then $G'$ can be mapped to $G_\infty$ such that vertices are sent to
vertices, edges to edges, and interior $3$-faces to $3$-faces,
such that adjacent $3$-faces in $G'$ are mapped to distinct $3$-faces in $G_\infty$.
\end{lemma}

Let $g$ be a map from $G'$ to $G_\infty$ as described in Lemma~\ref{C:map}.
Note that $g$ is not necessarily injective.
There is a pattern of vertices from $G_{\infty}$
that uses every seventh vertex (see the right side of Figure~\ref{F:infinite-grid}).
Let $D_\infty \subseteq V(G_{\infty})$ be the (infinite) set of vertices indicated in the figure.
For each vertex $v\in V(G')$, $g(v)\in D_\infty$ or $g(v)$ is adjacent to a vertex in $D_\infty$.  If $v\not\in V_T'$ then
the seven vertices of $N[v]$ map to the seven vertices of $N[g(v)]$; as this
includes one vertex of $D_\infty$, $v$ is dominated by $g^{-1}(D_\infty)=\{v\in V(G')\st g(v)\in D_\infty\}$.  Therefore $G'$ is dominated by the union
of $g^{-1}(D_\infty)$ and $V_T'$, or equivalently
the union of $g^{-1}(D_\infty)-V_T'$ and $V_T'$.
Let $D'= g^{-1}(D_\infty)-V'_T$ and
let $D$ be the union of $D'$ and $V(T)$.  By the construction
of $G'$ from $G$, $D$ is a subset of $V(G)$ that dominates every vertex of $G$.

In Subsection~\ref{SS:1}, we find upper bounds for $|D|$ in terms of $n$ and $|V(T)|$.
In~\cite{ArxivKingMJP}, where the maximum degree is at most~$6$, such a bound is much easier to find.

\subsection{Upper bounds for $|D|$}\label{SS:1}

Consider any distinct $v_1,v_2\in D'$ with $g(v_1)\not=g(v_2)$.
Since $D'\subseteq G'-V_T'$, $N[v_1]$ and $N[v_2]$ are $7$-vertex subsets of $G'$.
Then, by Lemma~\ref{C:map}, $g$ maps $N[v_1]$ and $N[v_2]$ bijectively to $N[g(v_1)]$ and $N[g(v_2)]$, respectively.
According to the right side of Figure~\ref{F:infinite-grid}, since $g(v_1)\not=g(v_2)$,
$N[g(v_1)]$ and $N[g(v_2)]$ are disjoint.
Then $N[v_1]$ and $N[v_2]$ must also be disjoint.
Therefore, if $v_1,v_2\in D'$ and $N[v_1]\cap N[v_2]\not=\emptyset$, then $g(v_1)=g(v_2)$.

\begin{lemma}\label{newlemma}
If $g(v_1)=g(v_2)$ in $G_\infty$ where $v_1,v_2\in V(G')$ are distinct vertices, and $u\in N[v_1]\cap N[v_2]$, then $\deg(u)>6$ and there are at least~5 edges between $uv_1$ and $uv_2$ in the rotation at $u$ in the drawing of $G'$.
\end{lemma}

\begin{proof}
If $v_1,v_2$ are adjacent in $G'$, then they are mapped to adjacent
(hence distinct) vertices by Lemma~\ref{C:map}.
So $v_1$ and $v_2$ are not adjacent, and $u\not\in \{v_1,v_2\}$.
The $3$-faces of $G'$ that are incident to $u$ will either form
a path or a $6$-cycle in the dual of $G'$.  According to the map $g$, the
images of the faces under $g$ will again be consecutive around $g(u)$,
so the number of these faces between $v_1$ and $v_2$ must be $6i$ for some
positive integer $i$.  The number of edges between $uv_1$ and $uv_2$ in the
rotation at $u$ is $6i-1\ge 5$, and thus $\deg(u)\ge (6i-1)+2\ge7$.
(See Figure~\ref{overlap}.)
\qed
\end{proof}

\begin{figure}[ht]
\begin{center}
\begin{pspicture}(4,2)
  \psdots(1,1)(3,1)
  \rput(2,-0.2){$u$}
  \psline(2,0)(1,1)
  \psline(2,0)(3,1)
  \psline(2,0)(2,1)
  \psline(2,0)(1.7,1)
  \psline(2,0)(1.3,1)
  \psline(2,0)(2.3,1)
  \psline(2,0)(2.7,1)
  \psline(2,0)(3,0.7)
  \psline(2,0)(1,0.7)
  \psline(1,1)(.7,1.3)
  \psline(3,1)(3.3,1.3)
  \psline(2.7,1)(3.3,1)
  \psline(3.3,1)(3,0.7)
  \psline(3,0.7)(3,1.3)
  \psline(2.7,1)(3,1.3)
  \psline(3.3,1.3)(3.3,1)
  \psline(3.3,1.3)(3,1.3)
  \psline(1,.7)(1,1.3)
  \psline(1.3,1)(.7,1)
  \psline(.7,1.3)(1,1.3)
  \psline(.7,1.3)(.7,1)
  \psline(.7,1)(1,.7)
  \psline(1.3,1)(1,1.3)
\end{pspicture}
\end{center}
\caption{Two vertices $v_1,v_2$ with $g(v_1)=g(v_2)$ and $u\in N[v_1]\cap N[v_2]$}
\label{overlap}
\end{figure}
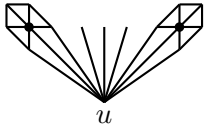

Let $t'$ be the number of vertices in $G'$ of degree greater than $6$, and for each vertex $v\in D'$
let $t'_v$ be the number of vertices in $N[v]$ with degree greater than $6$.

Consider any $v\in D'$, and suppose that $u\in N[v]\cap N[v']$ for some $v'\in D'$ with $v\not=v'$.
By the paragraph preceding Lemma~\ref{newlemma}, $g(v)=g(v')$.
Then, by Lemma~\ref{newlemma}, the degree of $u$ is greater than~$6$.
Therefore, vertices of degree at most~$6$ in $N[v]$ are not in
$N[v']$ for any $v'\in D'$ with $v\not=v'$.
Since $v\in D'\subseteq G'-V_T'$, $N[v]$ is a $7$-vertex subset of $G'$.
Hence, $N[v]$ contains $7-t_v'$ vertices of degree at most~$6$,
which are not in $N[v']$ for any $v'\in D'$ with $v\not=v'$.
Therefore,
$\sum_{v\in D'}(7-t_v')\le |V(G')|-t'$.
We can write \[|D'|= \frac{1}{6}\sum_{x\in D'}6 \le \frac{1}{6}\left({|V(G')|-t'+\sum_{v\in D'}(t_v'-1)}\right)\]
and soon, we will bound $\sum_{v\in D'}(t_v'-1)$.

If $\Delta(G)\le6$, then any vertex with degree greater than~6 in $G'$
must be in $V_T'$ and must, as a vertex in $G$, be a leaf of $T$.
If we also have $U=U_0=\{v\in V(G):\deg(v)\not=6\}$, then $T$ is a Steiner tree
for $U$, so every leaf of
$T$ has degree less than~6.  Therefore, if $\Delta(G)\le6$ and $U=U_0=\{v\in V(G):\deg(v)\not=6\}$,
then $t'$ and every $t_v'$ is zero, and the results of the previous paragraph
simplify to
$7|D'|\le |V(G')|$, or $|D'|\le \frac{1}{7}|V(G')|$.

To bound $\sum_{v\in D'}(t_v'-1)$, we make an auxilliary plane graph $H$:
Let $V(H)$ be the set of all vertices in $G'$ that have degree greater than $6$.
For each $v\in D'$, let $B[v]$ be the union of the six triangles
incident to $v$ in the embedding of $G'$ in the plane (including the interior and boundary of each triangle).
Then each $B[v]$ is a hexagon with $t_v'$ vertices of $H$ on its boundary.
For distinct $x,y\in D'$, the interiors of the hexagons $B[x]$ and $B[y]$
are disjoint; the hexagons may intersect on their boundaries.
Now, for each $v\in D'$ with $t'_v\ge2$, draw a $(t'_v-1)$-leaf star in $B[v]$ on the $t_v'$ vertices of $V(H)$ in $B[v]$,
such that the edges are drawn on the interior of $B[v]$.  (See Figure~\ref{F:H} for an example.)
Let $E(H)$ be the set of all such edges.
Then $H$ is a plane graph with $t'$ vertices and $\sum_{v\in D'}(t_v'-1)$ edges.

\begin{figure}[ht]
\begin{center}
\begin{pspicture}(2,1.6)
  \psset{unit=.4cm}
  \pspolygon[linestyle=dotted](3,0)(1,0)(0,1.732)(1,3.464)(3,3.464)(4,1.732)
  \psdots(1,0)(1,3.464)(3,3.464)(4,1.732)
  \rput(0.6,-0.4){$a$}
  \rput(4.4,1.732){$b$}
  \rput(0.6,3.9){$c$}
  \rput(3.4,4.05){$d$}
  \psline(1,0)(4,1.732)
  \psline(1,0)(1,3.464)
  \psline(1,0)(3,3.464)
\end{pspicture}
\end{center}
\caption{An example of a star drawn in $B[v]$ when $t_v'=4$}
\label{F:H}
\end{figure}
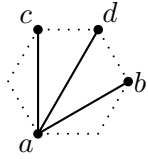

\begin{lemma} \label{L:outerplanar}
$H$ is an outerplanar graph with no multiple edges.
\end{lemma}

\begin{proof}
Assume there are two edges in $H$ sharing the same endpoints, say $x,y$, which means there are vertices $u,v\in D'$
such that $B[u]$ and $B[v]$ each contain both $x$ and $y$.  Let $Q$ be $G'$ restricted to the quadrilateral $xuyv$
and its interior, as shown in Fig~\ref{multiple edge}.  By Lemma~\ref{newlemma}, there are at least $5$ edges
incident to $x$ between the two edges $ux$ and $vx$ in the quadrilateral, so $\deg_Q(x)\ge7$; similarly $\deg_Q(y)\ge 7$.
All the vertices in the interior of $G'$ have degree $6$, and $\deg_Q(u)\ge 2,\deg_Q(v)\ge 2$,
so we have $\sum_{z\in Q}\deg(z)\ge 6(|V(Q)|-4)+2+2+7+7=6|V(Q)|-6$.  Also, $\sum_{z\in Q}\deg(z)=2|E(Q)|$, so
$|E(Q)|\ge 3|V(Q)|-3$.  However, $Q$ is planar, so Euler's formula yields $|E(Q)|\le 3|V(Q)|-6$.
This is a contradiction, so $H$ has no multiple edges.

The vertices of degree greater than~$6$ in $G'$ are all on the boundary of $G'$, so
all these vertices are incident to the outer face of $H$ as well.  These are the vertices of $H$, so $H$ is outerplanar.
\qed
\end{proof}

\begin{figure}[ht]
\begin{center}
\begin{pspicture}(5,4)
  \psset{unit=.7cm}
  \psdots(0,2)(4,2)(2,0)(2,4)
  \pspolygon(0,2)(2,0)(4,2)(2,4)
  \rput(2,4.4){$x$}
  \rput(2,-0.4){$y$}
  \rput(-0.4,2){$u$}
  \rput(4.4,2){$v$}
  \psline(2,0)(2,1)
  \psline(2,0)(1.8,1)
  \psline(2,0)(1.6,1)
  \psline(2,0)(2.2,1)
  \psline(2,0)(2.4,1)
  \psline(2,4)(2,3)
  \psline(2,4)(1.8,3)
  \psline(2,4)(1.6,3)
  \psline(2,4)(2.2,3)
  \psline(2,4)(2.4,3)
\end{pspicture}
\end{center}
\caption{The quadrilateral $xuyv$ and its interior in $H$}
\label{multiple edge}
\end{figure}
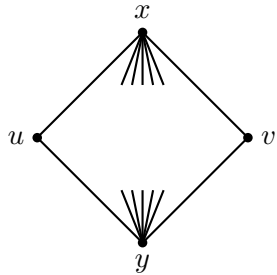

By Lemma~\ref{L:outerplanar}, $|E(H)|\le 2|V(H)|-3$, so $\sum_{v\in D'}(t_v'-1)\le 2t'-3$.
Then $|D'|\le \frac{1}{6}(|V(G')|-t'+2t'-3)=\frac{1}{6}(|V(G')|+t'-3)$.
Since $|V(G')|=n-|V(T)|+|V_T'|$ and $|D|=|D'|+|V(T)|$, we obtain
$|D|\le \frac{1}{6}(n+5|V(T)|+|V_T'|+t'-3)$.
Clearly, $t'\le |V_T'|$, so $|D|\le \frac{1}{6}(n+5|V(T)|+2|V_T'|-3)$.
When we defined $G'$, we made $\deg_T(v)$ copies of $v$ for all $v\in V(T)$, so $|V_T'|=\sum_{v\in T}\deg_T(v)=2|E(T)|=2|V(T)|-2$.
Therefore, \[|D|\le \frac{1}{6}(n+9|V(T)|-7).\]

Recall that if $\Delta(G)\le 6$ and $U=U_0=\{v\in V(G):\deg(v)\not=6\}$, then $|D'|\le \frac{1}{7}|V(G')|$.
Since $|V(G')|=n-|V(T)|+|V_T'|$ and $|V_T'|=2|V(T)|-2$, in this case we have
$|D|=|V(T)|+|D'|\le |V(T)|+\frac{1}{7}(n-|V(T)|+2|V(T)|-2)$, or
\[ |D| \le \frac{1}{7}(n+8|V(T)|-2).\]

Thus we have found upper bounds for $|D|$ in terms of $n$ and $|V(T)|$.
When this bound does not suffice, we will need to find
a different dominating set.

\subsection{When $G$ contains a large triangulated cylinder}\label{SS:2}

Within this subsection, we assume that $|U_0|>1$.

Recall that $P_0$ is a longest path in $T_0$ such that no internal vertices are in $U_0'$.
Let $v_1,v_2$ be the endpoints of the path $P_0$.
Let $x$ be a middle vertex of $P_0$, that is, let $x$ be a vertex on $P_0$ of distance $\lfloor{|P_0|/2}\rfloor$ from an
endpoint of $P_0$.
Let $N_i(x)$ be the set of vertices of $G$ with distance exactly $i$ from $x$,
let $N_i[x]$ be the set of vertices of $G$ with distance at most $i$ from $x$, and
let $G_i$ be the graph induced by $N_i[x]$.

The $d_U=0$ case of the following result was obtained in~\cite{ArxivKingMJP},
using mostly the same proof.

\begin{lemma}\label{L:P and N_i[x]}
$N_i(x)$ does not intersect $U$
for $i<\lfloor{|P_0|/2}\rfloor-d_U$.
\end{lemma}

\begin{proof}
Suppose that $j$ is the smallest index such that $N_j(x)$ intersects $U_0$.
Since each $u\in \Uo$ has distance at most $d_U$ to a vertex of $U_0$,
vertices of $N_i(x)$ with $i<j-d_U$ do not intersect $U$.

Let $u\in U_0\cap N_j(x)$ and
let $Q$ be an $x,u$-path of
length $j$ (which is in $G_j$\,).
There is a unique $x,u$-path in $T_0$;
without loss of generality, assume
that it contains $v_1$ (rather than $v_2$).
By the choice of $P_0$, deleting the interior of its $x,v_1$-subpath from
$T_0$ gives a 2-component graph that contains $U_0$.  We could then add $Q$
to that graph to obtain a connected graph that contains $U_0$, and let
$T_0'$ be a spanning tree of it.  Then,
$|V(T_0')|\le |V(T_0)|-(\lfloor{|P_0|/2}\rfloor-1)+(j-1)$. But $|V(T_0)|\le |V(T_0')|$ since
$T_0$ is a Steiner tree for $U_0$, so $j\ge \lfloor{|P_0|/2}\rfloor$.
\qed
\end{proof}

Let $r$ be minimum such that $G_r$ is not a triangulated hexagon.
Then $G_{r-1}$ accounts for $1+\sum_{i=1}^{r-1} 6i$ distinct vertices, so $n> 1+6r(r-1)/2>3(r-1)^2$.

\medskip

Next, we seek to understand the structure of $G_i$ for values $i\ge r$.  Our immediate goal is Lemma~\ref{just need it part 1};
Lemmas~\ref{L:1cycleABCDE} and~\ref{L:W'} help us obtain it.  In~\cite{ArxivKingMJP} the result stated in Lemma~\ref{just need it part 1}
was obtained for graphs of maximum degree at most~$6$, using a different proof.

From Subsection~\ref{S:def}, recall the definitions {\em outer degree sequence} of a walk that bounds a connected outerplane subgraph of $G$
and {\em type A, B, C, D, and E cyclic sequences}.

\begin{lemma}[King and Pelsmajer~\cite{ArxivKingMJP}]\label{L:1cycleABCDE}
Suppose that $x$ is a vertex in a plane triangulation $G$, and $r$ is the minimum such that $G_r$
(the graph induced by $N_r[x]$) is not a triangulated hexagon.
If every vertex in $N_r[x]$ has degree~6, then $N_r[x]$ contains a cycle $C'$ of length at most $2r+1$
of type A, B, C, D, or E.  Moreover, if $C$ is type A or E then $|C'|\in \{2r,2r+1\}$.
\end{lemma}

In \cite{ArxivKingMJP}, walks of type C and D appear when ``$q=1$'', type B appear when ``$q=2$'' and type A and E arise when ``$q=3$''.

\begin{lemma}\label{L:W'}
Suppose that $W=v_0,v_1,\ldots,v_m=v_0$ is a walk (indexed by the cyclic group $Z_m$)
that bounds an outerplane subgraph $H$ of a plane triangulation $G$ such that $W$ is oriented
counterclockwise, that is, with the exterior of $H$ always to its right.
Suppose that $W$ is type A, B, C, D, or E and suppose that every vertex of $W$ has degree~$6$ in $G$.

Then $W$ is a cycle of $G$.
Furthermore, the neighbors of $W$ on the interior of $H$ form a connected outerplane graph $H'$,
bounded by a counterclockwise walk $W'$ with length and type specified by Table~\ref{t:type}.
\end{lemma}

\begin{proof}
If $H$ has only one vertex $v$, then $\deg(v)=\rdeg_0(W)$.
If $H$ contains a leaf $v$, which is the $i$th vertex of $W$,
then $\deg(v)=1+\rdeg_i(W)$.
Since $\deg(v)=6$ for all $v\in V(H)$
and $\rdeg_i(W)\le 4$ for all $i\in Z_m$, both cases give a contradiction.

Suppose that $H$ contains a 2-connected leaf-block $B$ with cut-vertex $v$,
with $v$ appearing as the $i$th and $j$th vertex of $W$ as it enters and exits $B$.  Then $\deg_B(v)\ge 2$
and $\deg_H(v)\ge \deg_B(v)+1$, so $\rdeg_i(W) + \rdeg_j(W) \le \deg(v)-\deg_H(v) \le 6-3=3$.
Then the outer degree sequence of $W$ must be type~E with $\{\rdeg_i(W),\rdeg_j(W)\}=\{1,2\}$.
Also, the three previous inequalities must be equalities; in particular, $\deg_H(v)=3$.
Therefore, $v$ is incident to a single-edge (cut-edge) block of $H$, and it and $B$ are the only two blocks of $H$ that contain $v$.

Since the previous argument applies to any $2$-connected leaf-block of $H$, there is at most one $2$-connected leaf-block in $H$.
Since $H$ has no leaves, the block-cutpoint tree of $H$ has at most one leaf.
However, any nontrivial tree has at least two leaves, so the block-cutpoint tree must have only one vertex.
That is, $H$ must be $2$-connected.  Then $W$ is a cycle.

Since $W$ is a cycle, each vertex $v_i$ in $W$ is incident to exactly two edges in $W$ and $\rdeg_i(W)$ edges on the exterior of $W$.
Also $\deg(v_i)=6$, so $v_i$ is incident to exactly $4-\rdeg_i(W)$ edges on the interior of $W$.
(See Table~\ref{t:rdeg}, ignoring the last column.)

\begin{table}[ht]
\begin{center}
\begin{tabular}{c|c||c|c}
$\rdeg_i(W)$&number of such $i$&$4-\rdeg_i(W)$&$3-\rdeg_i(W)$\\
\hline
3&at most two&1&0\\
4&at most one&0&(none)\\
1&at most one&3&2\\
2&$m$, $m-1$ or $m-2$&2&1\\
\end{tabular}
\end{center}
\caption{The number of edges incident to $v_i$ on the exterior of $W$,
the frequency of each case, the number of edges incident to $v_i$ on the interior of $W$,
and the number of edges of $W'$ that form a triangle with $v_i$
(when $W$ is of type A, B, C, D, or E and $m=|W|$)}
\label{t:rdeg}
\end{table}

If $W$ is of type~C and $v_i$ is its vertex with $\rdeg_i(W)=4$,
then $v_i$ is incident to no edges on its interior, so $v_{i-1}v_iv_{i+1}$ is a triangle on the interior of $H$.
Removing $v_i$ from $W$ yields a cycle of length $m-1$ of type~D that bounds the outerplane graph $H-v_i$.
For convenience, rename $W$, $m$, and $H$ to be these values instead.
Now $W$ is of type A, B, D, or E.  Note that every value of $4-\rdeg_i(W)$ is $2$ except for at most one $3$
and up to two $1$s.

Suppose that $W$ has at least one chord on its interior.
(A {\em chord} of $W$ is an edge of $G-E(W)$ with endpoints in $W$.)
By the choice of $H$, we may assume
without loss of generality that $E(H)$
is the union of $E(W)$ and all chords of $W$ that lie in the interior of $W$.
Then the weak dual of $H$ is a nontrivial tree, with at least two
leaf-faces.
Let $v_iv_j$ be a chord incident to a leaf-face $f$ of the weak dual of $H$,
and let $B$ be the boundary cycle of $f$.  For any edge $xy$ in $B$, there is
a triangle $xyz$ in $G$ in $B\cup f$.  $B$ is chordless, so either $z$ is in $f$
or $B=xyz$.  If $B=xyz$, then $f$ is a face of $G$; letting $\{x,y\}=\{v_i,v_j\}$,
$z$ is incident to zero edges
on the interior of $W$ --- a contradiction.  Hence, $z$ is in $f$.  It follows
that every vertex of $B$, including $v_i$ and $v_j$, is incident to an edge in $f$.
At most one vertex of $W$ is incident to more than two edges on its interior.
Since $v_i$ and $v_j$ are each incident to an edge in $f$ and the edge $v_iv_j$,
without loss of generality we can assume that $v_i$ is incident to no other edges
in the interior of $W$.
Let $v_iv_jz$ be the unique triangle of $G$ that contains $v_iv_j$ and is not in $B\cup f$.
Then $v_iz$ must be an edge of $W$.  As before, $v_iv_jz$ does not bound a
leaf-face of the weak dual of $H$, because then $z$ would be incident to zero edges on the interior of $W$.
Therefore $v_jz$ must be a chord of $W$.  Hence, $v_j$ is the unique vertex of $W$ that
is incident to $3$ edges on the interior of $W$.
Recall that $H$ has a leaf-face $f'\not=f$.  Applying the same argument to $f'$,
we find that $f'$ is also incident to $v_j$, and that $v_j$ is incident to an edge in $f'$.
This edge is not a chord and it is not in $f$, so $v_j$ is incident to four edges in the
interior of $W$ --- a contradiction.
Thus, we have shown that $W$ has no chords in its interior.

For each $i\in Z_m$, the vertices adjacent to $v_i$ in the interior of $W$ form
a path $Q_i$, such that for each edge $xy$ in $Q_i$, $v_ixy$ bounds a face of $G$.
Let $W'$ be the closed walk obtained by concatenating these paths.  By construction,
$W'$ bounds an outerplane graph, oriented in the counterclockwise direction.
$Q_i$ has $4-\rdeg_i(W)$ vertices, so its length is $3-\rdeg_i(W)$.
See Table~\ref{t:rdeg} on the right.  Thus, each $v_i$ in $W$ yields one edge in $W'$,
except if $\rdeg_i(W)$ is~$3$ or~$1$, in which case $v_i$ yields $0$ or~$2$ edges in $W'$.
This shows that $W'$ has the length specified by Table~\ref{t:type} when $W$ is type A, B, D, or E.

\begin{table}[ht]
\begin{center}
\begin{tabular}{c||c|c}
Type of $W$&length of $W'$&Type of $W'$\\
\hline
A&$m$&A\\
B&$m-1$&B\\
C&$m-3$&C\\
D&$m-2$&C\\
E&$m$&E\\
\end{tabular}
\end{center}
\caption{How the type and length of $W$ determine the type and length of $W'$ (where $m=|W|$)}
\label{t:type}
\end{table}

For each $v_i'$ in $W'$, $\rdeg_i(W')$ is the number of paths $Q_i$ that contains $v_i'$.
Thus, when $Q_i$ and $Q_{i+1}$ have length at least one, they meet at a vertex $v_i'$ with
$\rdeg_i(W')=2$; when $Q_i$ has length two (when $W$ is type E) it yields on vertex $v_i'$ with $\rdeg_i(W')=1$;
when $|Q_i|=|Q_{i+1}|=0$ (type D) or $|Q_i|=0$ (type B), there is a vertex $v_i'$ with
$\rdeg_i(W')=4$ or $\rdeg_i(W')=3$.  Altogether, this shows that $W'$ has the type indicated in
Table~\ref{t:type}, when $W$ is type A, B, D, or E.

When $W$ is type~C, then it was replaced by a walk of type~D and length $m-1$, so by the previous
two paragraphs, $W'$ should have length $(m-1)-2$ and be type~C, as indicated in Table~\ref{t:type}.
\qed
\end{proof}

When $N_{3r+1}[x]$ does not intersect $U$, we will produce a
{\em triangulated cylinder} (recall definition from Subsection~\ref{S:def}).

\begin{lemma}\label{just need it part 1}
Suppose that $G$ is a plane triangulation, with $P_0$, $x$, and $r$ defined as before.
If $N_{3r+1}[x]$ does not intersect $U$,
then $G$ contains a $(w,\ell)$-cylinder $H$ with no interior vertices in $U$, such that $w\in\{2r,2r+1\}$ and $\ell \ge 2 (\lfloor{|P_0|/2}\rfloor -r-d_U-1)$.
\end{lemma}

\begin{proof}
Since $N_{3r+1}[x]$ does not intersect $U$, every vertex in $N_{3r+1}[x]$ has degree~6.
By Lemma~\ref{L:1cycleABCDE}, $N_r[x]$ contains a cycle $W_0$ of length at most $2r+1$
of type A, B, C, D, or E.  $W_0$ bounds an outerplanar subgraph of $G$, and it can be
oriented counterclockwise.  For any $i\ge 1$, let $N_i^*$ be the set of vertices in the
interior of $W_0$ at distance exactly $i$ from $W_0$.
Let $j$ be minimum such that $N_j^*=\emptyset$ or $N_j^*$ intersects $U$.
Since $W_0$ is in $N_r[x]$, $j>0$.
Then we can repeatedly apply Lemma~\ref{L:W'}
for $i=0,\ldots,j-1$ with $W=W_i$, which proves that $W_i$ is a cycle of type A, B, C, D, or E
and produces the closed (nonempty) walk $W_{i+1}=W'$ on vertex set $N_{i+1}^*$.
Therefore, $W_j$ is nonempty, so $N_j^*$  must contain a vertex of $U$.

According to Table~\ref{t:type}, if $W_0$ is type B, C, or D,
then for all $1\le i\le 2r+1$, every $W_i$ is type B or C.
Then $|W_{i+1}|\le |W_i|-1$ for all $0\le i\le j-1$, so $|W_j|\le |W_0|-j$.
$W_j$ is nonempty and $|W_0|\le 2r+1$, so $0< 2r+1-j$.
Since $W_0$ is in $N_r[x]$, $N_i^*\subseteq N_{r+i}[x]$ for all $i\ge 1$.
Then $N_j^*\subseteq N_{r+j}[x]\subseteq N_{3r+1}[x]$.  However, $N_{3r+1}[x]$
does not intersect $U$, so this is a contradiction.
Thus we may assume that $W_0$ is type A or E, in which case
$W_i$ is the same type and length as $W_0$ for all $1\le i\le j$, and
$|W_0|\in \{2r,2r+1\}$ by Lemma~\ref{L:1cycleABCDE}.  Let $w=|W_0|$.
$W_i$ is a $w$-cycle for all $0\le i\le j-1$.

For any $i\ge 1$, let $N_{-i}^*$ be the set of vertices in the exterior of $W_0$
at distance exactly $i$ from $W_0$, and let $k$ be the minimum such that
$N_{-k}^*=\emptyset$ or $N_{-k}^*$ contains a vertex of degree not equal to $6$.
Now, reembed $G$ in the plane such that the interior and exterior of $W_0$ are switched
and repeat the previous argument.  This yields $w$-cycles $W_{-i}$ on $N_{-i}^*$ for all
$0\le i\le k-1$.  The cycles $W_i$ for $-(k-1)\le i\le j-1$ and the edges between
consecutive cycles form a $(w,\ell)$-cylinder with $w=|W_0|$ and $\ell=j+k-2$.

$N_j^*$ and $N_{-k}^*$
each contain a vertex of $U$, so $N_{r+j}[x]$ and $N_{r+k}[x]$ do too, since $W_0$ is in $N_r[x]$.
By Lemma~\ref{L:P and N_i[x]}, $r+j\ge \lfloor{|P_0|/2}\rfloor-d_U$ and $r+k\ge \lfloor{|P_0|/2}\rfloor-d_U$.
Then $\ell \ge j+k-2\ge 2 (\lfloor{|P_0|/2}\rfloor -r-d_U-1)$.
\qed
\end{proof}

The following results are proved in~\cite{ArxivKingMJP}.

\begin{lemma}[King and Pelsmajer~\cite{ArxivKingMJP}] \label{just need it part 2}
Suppose that $G$ is a plane triangulation and $\Delta(G)\le 6$.
If $H$ is a $(w,\ell)$-cylinder in $G$ with $\ell$ maximal,
then $G-V(H)$ has at most $w(w-1)$ vertices.
\end{lemma}

\begin{lemma}[King and Pelsmajer~\cite{ArxivKingMJP}] \label{just need it part 3}
Suppose that $H$ is a $(w,\ell)$-cylinder.  Then $H$ has $w(\ell+1)$ vertices and
there is a set of $\lceil{\frac{\ell}{7}}\rceil(w+2)$ vertices on $H$ that dominates all
vertices on the interior of $H$.
\end{lemma}

We will see that the size $\lceil{\frac{\ell}{7}}\rceil(w+2)$ is efficient enough --- that is, it uses roughly one-sixth of
the vertices on the $(w,\ell)$-cylinder or less --- when $w\ge 12$.  However, we must prove a new result for the cases when $w$ is small.

\begin{lemma} \label{L:new cylinder}
If $H$ is a $(w,\ell)$-cylinder and $3\le w\le 12$, then $H$ contains a set of at most $\frac{|V(H)|}{6}+12=\frac{1}{6}w(\ell+1)+12$ vertices
that dominates the interior vertices of $H$.
\end{lemma}

\begin{proof}
Let $Z=\{z_{a,b} : a\in Z_w, 0\le b\le \ell\}$ be the vertex set of $H$, and note that $|Z|=w(\ell+1)$.
Let $Z'=\{z_{a,b} : a\in Z_w, 0< b< \ell\}$, the vertices on the interior of $H$.

For all integers $w,k$ with $3\le w\le 12$ and $0\le k< w$, we will give an integer $m=m(w,k)$
and a set $S\subseteq \{ z_{a,b} : a\in Z_w, 0\le b<m \}$ such that (i)
every $z_{a,b}\in Z'$ is dominated by some $z_{c,d}\in Z$ such that $z_{c,d\bmod m}\in S$, and
(ii) $|S| \le \min(\frac{1}{6}mw,12)$.
If we have such $m$ and $S$,
then $S_Z=\{z_{a,b}\in Z : z_{a,b\bmod m}\in S\}$ is a set of size at most $|S|\lceil{(\ell+1)/m}\rceil$ in $Z$ that dominates $Z'$.
Since $|S|\lceil{(\ell+1)/m}\rceil\le |S|(\frac{\ell}{m}+1)\le \frac{1}{6}w \ell +12$, this will finish the proof.

Thus, it remains to find such $m,S$ for all $w,k$ such that $3\le w\le 12$ and $0\le k< w$.
Recall that we may assume that $0\le k\le w/2$.
In each case of the proof, we describe $S$ via a figure where $z_{a,b}$ is located on row $a$ and column $b$, showing
rows $0$ to $w$ (row $0$ and row $w$ are identified) and columns $0$ to at least $m$.
When $b=0\equiv w\,(\bmod\,w)$, then $z_{a,b}$ is shown twice, but once as a hollow dot (for example, in Figure~\ref{F:w=6}).
The figure will make it clear that
$S_Z$ dominates $Z'$ as desired.

We begin with two general cases: when $w$ is a multiple of $2$ or $3$.

Consider any case where $w$ is a multiple of $3$.
No matter what $k$ is, any three consecutive rows of $Z'$ can be dominated by taking every other vertex on the middle row.
(For example, see Figure~\ref{F:w=6} on the left).
Thus, we have $m,S$ with $m=2$ and $|S|=w/3$. Clearly $S_Z$ dominates $Z'$
and $|S| = \frac{1}{6}mw \le 12$, so this suffices.

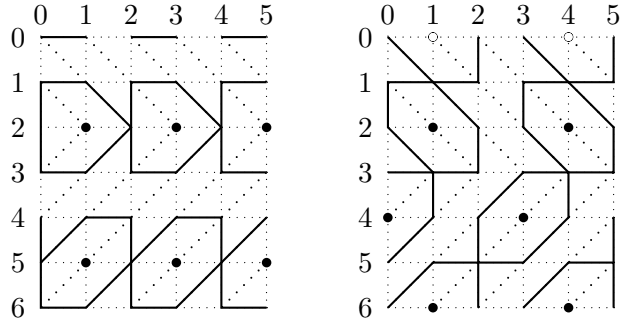
\begin{figure}[ht]
\begin{center}
\begin{tabular}{ccc}
\begin{pspicture}(3.0,4.0)
\psset{unit=.6cm}
\multido{\iA=0+1}{6}{%
    \rput(\iA,6.5){$\iA$}}
\multido{\iA=0+1,\iB=6+-1}{7}{%
    \rput(-0.5,\iB){$\iA$}}
  \psgrid[subgriddiv=1,griddots=6,gridlabels=0pt,gridwidth=.5pt](5,6)
  \multido{\rA=0+1,\rB=1+1}{5}{%
    \psline[linestyle=dotted](\rA,0)(\rB,1)
    \psline[linestyle=dotted](\rA,1)(\rB,2)
    \psline[linestyle=dotted](\rA,2)(\rB,3)
    \psline[linestyle=dotted](\rA,3)(\rB,4)
    \psline[linestyle=dotted](\rA,5)(\rB,4)
    \psline[linestyle=dotted](\rA,6)(\rB,5)}
  \multido{\rA=1+2}{3}{%
    \psdots(\rA,1)
    \psdots(\rA,4)}
  \multido{\rA=0+2,\rB=1+2,\rC=2+2}{2}{%
    \pspolygon(\rA,0)(\rB,0)(\rC,1)(\rC,2)(\rB,2)(\rA,1)
    \pspolygon(\rA,3)(\rB,3)(\rC,4)(\rB,5)(\rA,5)(\rA,4)
    \psline(\rA,6)(\rB,6)}
  \psline(5,0)(4,0)(4,1)(5,2)
  \psline(5,3)(4,3)(4,5)(5,5)
  \psline(4,6)(5,6)
  \psline(0,1)(0,2)
\end{pspicture}
& \qquad\qquad &
\begin{pspicture}(3.0,4.0)
\psset{unit=.6cm}
\multido{\iA=0+1}{6}{%
    \rput(\iA,6.5){$\iA$}}
\multido{\iA=0+1,\iB=6+-1}{7}{%
    \rput(-0.5,\iB){$\iA$}}
  \psgrid[subgriddiv=1,griddots=6,gridlabels=0pt,gridwidth=.5pt](5,6)
  \multido{\rA=0+1,\rB=1+1}{5}{%
    \psline[linestyle=dotted](\rA,0)(\rB,1)
    \psline[linestyle=dotted](\rA,1)(\rB,2)
    \psline[linestyle=dotted](\rA,2)(\rB,3)
    \psline[linestyle=dotted](\rA,4)(\rB,3)
    \psline[linestyle=dotted](\rA,5)(\rB,4)
    \psline[linestyle=dotted](\rA,6)(\rB,5)}
  \psdots(1,0)(4,0)(0,2)(3,2)(1,4)(4,4)
  \psdots[dotstyle=o](1,6)(4,6)
  \psline(0,0)(1,1)(2,1)(2,0)
  \psline(3,0)(4,1)(5,1)(5,0)
  \psline(0,3)(1,3)(1,2)(0,1)
  \pspolygon(2,2)(2,1)(3,1)(4,2)(4,3)(3,3)
  \psline(5,1)(5,2)
  \pspolygon(0,4)(1,3)(2,3)(2,4)(1,5)(0,5)
  \pspolygon(3,4)(4,3)(5,3)(5,4)(4,5)(3,5)
  \psline(0,6)(1,5)(2,5)(2,6)
  \psline(3,6)(4,5)(5,5)(5,6)
\end{pspicture}

\end{tabular}
\end{center}
\caption{Dominating sets when $w=6$:
an example with $m=2$ and $k=2$, with $3m$ columns shown (left)
and an example with $m=3$ and $k=3$, with $2m$ columns shown (right)}
\label{F:w=6}
\end{figure}

Next, consider any case where $w$ is even.  Let $m=3$, and $S$ will contain one vertex from
each even-indexed row.  Then $|S|=\frac{1}{2}w = \frac{1}{6}mw \le 6$.
Clearly, $S_Z$ will dominate all vertices in each even-indexed row of $Z'$.
It remains to show that for any $k$, $S$ can be constructed so that
$S_Z$ dominates all vertices in all odd-indexed rows of $Z'$.

For each row $2i$, $S$ contains
either $z_{2i,0}$, $z_{2i,1}$, or $z_{2i,2}$, and this ``offset''
determines the entire pattern on that row in $S_Z$.
(For example, see Figure~\ref{F:w=6} on the right.)
Once the offset is
chosen for row $2i$, two of the three possible
offsets for row $2i+2\,(\bmod\,w)$ will ensure that all vertices in row $2i+1\,(\bmod\,w)$ of $Z'$ are dominated.
Thus, starting with $z_{2,1}$ in $S_Z$ for row $2$, the offsets for
rows $2i$ with $1\le i\le w/2-1$ can be chosen so that all vertices in rows $3,\ldots,w-3$ of $Z'$ are dominated.
Finally, of the three possible offsets for row $0$,
two will ensure that all vertices in row $w-1$ of $Z'$ are dominated, and
two will ensure that all vertices in row $1$ of $Z'$ are dominated; hence, there is an offset for row $0$
so that both rows are dominated.

It remains to consider the cases $w=5,7,11$, for all $0\le k< w$.
We consider each of these cases separately, giving a figure that shows $S$ in $\{z_{a,b} : 0\le b< m\}$ in the appendix
and noting that
$|S|\le 12$ and $\frac{|S|}{mw}\le \frac{1}{6}$ in Table~\ref{T:table}.  This completes the proof of Lemma~\ref{L:new cylinder}.
\qed
\end{proof}

\begin{table}[ht]
\begin{center}
\begin{tabular}{|c|c||c||c|c|c|}
\hline
$w$&$k$&Figure&$|S|$&$m$&$|S|/(mw)$\\
\hline
5&0&8&4&5&4/25\\
5&1&8&5&7&1/7\\
5&2&8&5&6&1/6\\
7&0&9&7&7&1/7\\
7&1&9&8&7&8/49\\
7&2&10&8&7&8/49\\
7&3&10&8&7&8/49\\
11&0&11&12&7&12/77\\
11&1&11&12&7&12/77\\
11&2&11&9&5&9/55\\
11&3&12&12&7&12/77\\
11&4&12&12&7&12/77\\
11&5&12&12&7&12/77\\
\hline
\end{tabular}
\end{center}
\caption{All cases with $w=5,7,11$}
\label{T:table}
\end{table}

At this point, we break the argument into two proofs.

\subsection{Finishing the proof of Theorem~\ref{T:n/6}}\label{SS:3}

Suppose that $\Delta(G)\le 6$ and $U=U_0=\{v\in V(G):\deg(v)\not=6\}$.  Then $T=T_0$, $|\Uo|=d_U=0$,
$|U_0|=|U|\le 12$, and
there is a dominating set $D$ of $G$ with $|D|\le \frac{n+8|V(T)|-2}{7}$.
Also, $|U_0|>1$ because $|U|\ge 4$, so
we have $P_0$ and $r$ with
$n>3(r-1)^2$, and Lemma~\ref{L:P and N_i[x]} applies.

Let $c=1.05\times 10^7$.
If $|D|\le n/6+c$ then we are done, so we may assume that
$\frac{1}{7}\left(n+8|V(T)|-2\right)> n/6+c$, or equivalently,
$n < 48|V(T)|-42c-12$.  $|V(T)|=|V(T_0)|\le (2|U_0|-3)|P_0|+1 \le 21|P_0|+1$, so
$n <1008|P_0|-42c+36$ and $|P_0|> \frac{n-36+42c}{1008}$. And since $c=1.05\times 10^7$,
$|P_0|>\frac{42c-36}{1008}>42\times 10^4$.

\smallskip\noindent
{\bf Claim:}
{\em Every vertex in $N_{3r+1}[x]$ has degree~$6$.}

By Lemma~\ref{L:P and N_i[x]}, it is true if $3r+1<\lfloor |P_0|/2 \rfloor$.
Suppose that it is false.  Then $3r+1\ge \lfloor |P_0|/2 \rfloor \ge (|P_0|-1)/2$,
so $r\ge (|P_0|-3)/6$.
Since $n> 3(r-1)^2$, we get $n> \frac{1}{12}(|P_0|^2-18|P_0|+81)$.
Then $n <1008|P_0|-42c+36$ gives $12(1008|P_0|-42c+36) > |P_0|^2-18|P_0|+81$,
or equivalently, $12114 +(-504c+351)/|P_0| >|P_0|$.  Which contradicts $|P_0|>42\times 10^4$.
Thus the claim is proved.

\medskip

Now we may apply Lemma~\ref{just need it part 1} to obtain a $(w,\ell,k)$-cylinder $H$ with $\ell$ maximized.
By Lemmas~\ref{just need it part 1} and~\ref{just need it part 2}, $w=2r$ or $w=2r+1$, $\ell \ge |P_0|-2r-3$,
and $n\le w(\ell+1)+w(w-1)=w(\ell+w)$.  By Lemma~\ref{just need it part 3}, $H$ contains a set $S_H$ of at most $\lceil\frac{\ell}{7}\rceil(w+2)$ vertices that dominates its interior.  $V(G)-V(H)$ dominates itself and the
boundary of $H$, so
if we add $V(G)-V(H)$,
we get a set that dominates $G$; it has size at most $\lceil\frac{\ell}{7}\rceil(w+2)+n-w(\ell-1)$. We are done if this is at most $\frac{n}{6}+c$, so we may assume that
$\lceil\frac{\ell}{7}\rceil(w+2)+n-w(\ell-1)> \frac{n}{6}+c$.  Since $\lceil\frac{\ell}{7}\rceil\le \frac{\ell+6}{7}$ and $n\le w(\ell+w)$, we get
$w(\ell-1)-\frac{(\ell+6)(w+2)}{7}+c <\frac{5}{6}n \le \frac{5}{6}w(\ell+w)$.
It follows that $(w-12)\ell < 35w^2+78w+72-42c$.

First, consider the case $w\ge 13$.
Then $\ell < (35w^2+78w+72-42c)/(w-12) =
35w+498+\frac{6048-42c}{w-12}$.  Since $c>6048$ and $w\le 2r+1$, we have
$\ell< 35w+498\le 70r+533$. Since $\ell\ge |P_0|-2r-3$, we have $|P_0|< 72r+536$. With $r<1+\sqrt{n/3}$ and $|P_0|> \frac{n-36+42c}{1008}$, we can obtain
$n-36+42c - 1008(24\sqrt{3n}+608)<0$.
Let $f(x)=x^2-1008\cdot24\sqrt{3}x-1008\cdot 608-36+42c$; then $f(\sqrt{n})<0$.
Since $f(x)$ is a quadratic function with a positive quadratic term, $f(x)$ must have two roots.  Therefore
$(1008\cdot 24\sqrt{3})^2-4(-1008\cdot 608 + 42c-36)> 0$, which contradicts $c=1.05\times 10^7$ (but not by much, which explains our choice of $c$).
This completes the case $w\ge13$.

Next, suppose that $3\le w\le 12$.  By Lemma~\ref{L:new cylinder}, $H$ contains a set of size $S$ that dominates all vertices on its interior of $H$ with $|S|\le |V(H)|/6+12$.
If we add $V(G)-V(H)$ and all the vertices of the boundary of $H$ to $S$,
we get a set that dominates $G$; its size is at most $|V(H)|/6+12 +w(w-1)+2w$.
Since $w\le 12$ and $H\subseteq G$, we have
$|V(H)|/6+12 +w(w-1)+2w\le n/6+12+132+24<n/6+c$, as desired.
\qed

\subsection{Finishing the proof of Lemma~\ref{L:sphere}}\label{SS:4}

Recall that there is a dominating set $D$ of $G$ with $|D|\le \frac{n+9|V(T)|-7}{6}$.

If $|U_0|=1$, then $|V(T)|=1+|\Uo|$, so
\[
|D|\le \frac{n+9|\Uo|+2}{6} = \frac{n}{6}+ \frac{3}{2}|\Uo| + \frac{1}{3},
\]
as desired.  Thus, we may assume that $|U_0|>1$.

We have $P_0$ and $r$ with
$n>3(r-1)^2$.  Also, by Lemma~\ref{L:P and N_i[x]}, every vertex in $N_i[x]$ has degree~6
and is not in $U$ if  $i< \lfloor{|P_0|/2}\rfloor-d_U$.

\smallskip\noindent
{\bf Case 1} $|P_0|< 2\sqrt{3n}+2d_U+9$.

Then, since $|V(T_0)|\le (2|U_0|-3)|P_0|+1$ and $|V(T)|\le |V(T_0)|+ |\Uo|$,
we have $|V(T)|< (2|U_0|-3)(2\sqrt{3n}+2d_U+9)+|\Uo|+1$.  $G$ has a dominating set $D$
with $|D|\le \frac{n+9|V(T)|-7}{6}$.  Then
\[ |D| <
\frac{n}{6}+
\frac{3}{2}(2|U_0|-3)(2\sqrt{3n}+2d_U+9)
+\frac{3}{2}|\Uo|+ \frac{1}{3}.
\]

\noindent
{\bf Case 2} $|P_0|\ge 2\sqrt{3n}+2d_U+9$.

Since $n>3(r-1)^2$, we get $\sqrt{3n}>3(r-1)$, which
yields $|P_0|> 6r+2d_U+3$, so $|P_0|\ge 6r+2d_U+4$.
Then $\lfloor{|P_0|/2}\rfloor -d_U> 3r+1$, so by Lemma~\ref{L:P and N_i[x]},
every vertex in $N_{3r+1}[x]$ has degree~6 and is not in $U$.
By Lemma~\ref{just need it part 1}, $G$ has a $(w,\ell)$-cylinder $H$
with no interior vertices in $U$, such that $w\in\{2r,2r+1\}$ and $\ell \ge 2 (\lfloor{|P_0|/2}\rfloor -r-d_U-1)$.
Since $\lfloor{|P_0|/2}\rfloor \ge \frac{1}{2}(|P_0|-1)$ and
$\sqrt{3n}>3(r-1)$, we get
$\ell > 2\left(\frac{1}{2}(2\sqrt{3n}+2d_U+9-1) - \frac{1}{3}\sqrt{3n} -1 -d_U -1\right) = 4(1+\sqrt{n/3})$.

By Lemma~\ref{L:new cylinder}, the interior of the triangulated cylinder can be dominated by a set $S_Z$ of at most $\frac{w(\ell+1)}{6}+12$ vertices if $3\le w\le 12$.
In this case, $12\le 4w$, so we get $|S_Z|< \frac{w\ell}{6}+5w$.
If $w\ge 13$, then by Lemma~\ref{just need it part 3},
the interior of the triangulated cylinder can be dominated by a set $S_Z$ of at most $\lceil{\frac{\ell+1}{7}}\rceil(w+2)$ vertices.
Since $\lceil{\frac{\ell+1}{7}}\rceil\le \frac{\ell}{7}+1$, in this case we have
$|S_Z|\le (\frac{\ell}{7}+1)(w+2)=\frac{w\ell}{6}-\frac{(w-12)\ell}{42}+(w+2)< \frac{w\ell}{6}+2w$.
Thus, we have  $|S_Z|< \frac{w\ell}{6}+5w$ for all $w\ge3$.

In order to apply induction on $n$, we delete the $w(\ell-1)$ interior vertices of $(w,\ell)$-cylinder $H$. Let $C_1$ and $C_2$ be the boundary cycles of $H$; these now bound holes in the surfaces. Recall that every $w$-cycle on the cylinder has the exact same pattern of turns: none, or exactly one right turn and exactly one left turn which are at the same places around each cycle.  Thus we can identify $C_1$ and $C_2$ such
that corresponding turns are matched to each other.  Thus,
when two vertices of degree $6$ are identified, the resulting vertex will have degree $6$.
In this way, identify each pair of corresponding vertices $v_1,v_2$ from $C_1,C_2$ to get
a new vertex $v^*$ and a new $w$-cycle $C^*$.  This creates a new plane (or sphere)
triangulation $G^*$ with $n-w\ell$ vertices.

Still using $v_1,v_2,v^*$ to represent corresponding vertices on $C_1,C_2,C^*$,
we define disjoint subsets $U_0^*, \Us$ of $V(G^*)$
as follows: If $v_1$ or $v_2$ is in $U_0$, then put $v^*$ in $U_0^*$; also, let
$U_0^*\setminus V(C^*)=U_0\setminus V(H)$.  Note that $|U_0^*|\le |U_0|$.
If $v_1$ or $v_2$ is in $\Uo$ and $v^*\not\in U_0^*$, then put $v^*$ in $\Us$;
also, let $\Us\setminus V(C^*)=\Uo\setminus V(H)$. Note that $|\Us|\le |\Uo|$.
Let $U^*=U_0^* \cup \Us$; note that $U^*\setminus V(C^*)=U\setminus V(H)$
and that for each $v^*\in U^*\cap V(C^*)$, $v_1$ or $v_2$ is in $V(C)\cap U$.

We wish to show that
Definition~\ref{D:U} is satisfied by $G^*$ with $U^*,U_0^*,\Us,d_U$.

A vertex of degree not equal to 6 in $G^*$ is also a vertex of degree other than 6 in $G$
if it is not in $V(C^*)$, and if $v^*\in V(C^*)$ does not have degree 6 then at least one
of $v_1$ or $v_2$ does not have degree 6.  Therefore every vertex of degree not equal to 6
in $G^*$ is in $U^*$.

Any vertex $x\in \Us$ corresponds to a vertex $y\in \Uo$
(either $x=v^*\in V(C^*)$ and $y\in\{v_1,v_2\}$, or $x\not\in V(C^*)$ and $y=x$).
There is a path $P$ in $G$ from $y$ to $U_0$ of length at most $d_U$.  The vertices
of $P$ that lie in $H$ can be replaced by vertices on $C^*$ in a natural way so that we
get a walk in $G^*$ from $x$ to $U_0^*$ of length at most $|P|$.  Therefore, any vertex
in $\Us$ has distance at most $d_U$ in $G^*$ to $U_0^*$.

Since $U$ does not intersect the interior of $H$, each component of $G[U]$ becomes a
connected subgraph of $G^*[U^*]$ that contains at least one vertex of $U_0^*$.
Each component of $G^*[U^*]$ is the union of some of these subgraphs, so it also
intersects $U_0^*$.

Definition~\ref{D:U} is satisfied for $G^*$ with $U^*,U_0^*,\Us,d_U$,
so we can apply induction.  We get a dominating set $D^*$ that contains $U^*$ such that
\[|D^*| \le \frac{n-w\ell}{6} +
3(|U_0^*|-1)(2\sqrt{3(n-w\ell)}+2d_U+9) +
\frac{3}{2}|\Us| +
\frac{1}{3}.
\]

Since $|U_0^*|\le |U_0|$ and $|\Us|\le |\Uo|$, we get
\[|D^*| \le
\frac{n-w\ell}{6} +
3(|U_0|-1)(2\sqrt{3(n-w\ell)}+2d_U+9) +
\frac{3}{2}|\Uo| +
\frac{1}{3}.
\]

Temporarily set $x$ so that $\sqrt{n}-\sqrt{n-w\ell}= xw$.
Then $\sqrt{n}-xw=\sqrt{n-w\ell}$, so
$n+x^2w^2-2xw\sqrt{n}=n-w\ell$.  Then
$x^2w+\ell =2x\sqrt{n}$.
Since $\ell>4(1+\sqrt{n/3})$ and $x^2w>0$, we get
$4(1+\sqrt{n/3})<2x\sqrt{n}$.
Then $x>2\sqrt{1/3}$, so $\sqrt{n}-\sqrt{n-w\ell}> 2\sqrt{1/3}w$
and $2\sqrt{3(n-w\ell)} < 2\sqrt{3n} -4w$.
It follows that
\[|D^*|< \frac{n-w\ell}{6} +
3(|U_0|-1)(2\sqrt{3n}-4w+2d_U+9) +
\frac{3}{2}|\Uo| +
\frac{1}{3}.
\]

$G$ is dominated by the union of $D^*-V(C^*)$, $S_Z$, and $V(C_1)\cup V(C_2)$; this set
$D$ has size at most $|D^*|+(\frac{wl}{6}+5w)+2w$.  Since $|U_0|\ge 2$, we have
$3(|U_0|-1)(-4w)+5w+2w\le -5w<0$.  Thus, we get
\[ |D|<
\frac{n}{6} +
3(|U_0|-1)(2\sqrt{3n}+2d_U+9) +
\frac{3}{2}|\Uo| +
\frac{1}{3}
\]

Thus, whether $|U_0|\le 1$, or whether we are in one of the two cases where $|U_0|>1$, we obtain a dominating set $D$
for $G$ of the desired size.  This finishes the proof of the Lemma~\ref{L:sphere}.
\qed

\section{Small non-contractible cycles in non-orientable surfaces}\label{S:non-orientable}

In this section we prove Theorem~\ref{T:non}.

First we prove it for any triangulation $G$ on the projective plane $N_1$.
Given $G$ on $N_1$, let $C$ be a minimum-length non-contractible cycle.
$C$ must be one-sided.
Cut along $C$ and double $C$ alongside the cut, as in the $C$-derived construction from
Section~\ref{S:def}, but do not add the disk,
nor the extra vertex that goes in the disk.  This yields a
triangulated disk $G'$ bounded by a cycle $C'$ of length $2|C|$.
Label the vertices of $C'$ in clockwise order, as
$v_0,v_1,\ldots,v_{|C|}=v_0',v_1',\ldots,v_{|C|}'=v_0$;
then $v_j$ and $v_j'$ (for any $0\le j\le |C|$)
are copies of the same vertex in $C$.
Let $m=\lfloor{|C|/2}\rfloor$, let $x=v_m$, and for all $j\ge0$,
let $V_j$ be the set of vertices $v$ in $G'$ such that the distance $d(v,x)=j$.

\begin{lemma}\label{proj-plane}
$V_j$ contains a path $P_j$ from $v_{m-j}$ to
$v_{m+j}$ of length at least $2j$,
for all $j$ with $0\le j\le m$.
\end{lemma}

\begin{proof}
Any path $P$ in $G'$ between opposite vertices $v_j,v_j'$ of $C$ corresponds to
non-contractible cycle in $C$; then by the choice of $C$, the length of $P$ is at least $|C|$.
Recall that for any vertices $u,v$, any $u,v$-walk contains a $u,v$-path.

Suppose that $W$ is a $v_i,v_j$-walk (or path) in $G'$ with $0\le i<j\le m$.
Then $v_0,\ldots,v_i,W,v_j,\ldots,v_{|C|}=v_0'$ is
is a $v_0,v_0'$-walk in $G'$, which contains a $v_0,v_0'$-path $P$ in $G'$.
Since the length of $P$ must be at least $|C|$, the length of $W$ must be at least $j-i$.

For $0\le j\le m$, we can apply the previous observation where the indices are
$(m-j,m)$ or $(m,m+j)$ and conclude that $v_{m\pm j}\not\in V_i$ for any $i<j$.
Since $v_m,v_{m-1},\ldots,v_{m-j}$ and $v_m,v_{m+1},\ldots,v_{m+j}$ are paths of
length $j$, we have $v_{m\pm j}\in V_j$ for all $0\le j\le m$.
If we apply the same observation where the indices are $m-j$ and $m+j$,
we can conclude that any $v_{m-j},v_{m+j}$-walk (or path) in $G'$
has length at least $2j$, for any $0\le j\le m$.

Thus, it remains to show that $G'$ contains a $v_{m-j},v_{m+j}$-path $P_j$
with $V(P_j)\subseteq V_j$, for all $0\le j\le m$.  We prove this by induction.
It is trivial for $j=0$ since $V_0=\{x\}$.
Assume that it is true for fixed $j$ with $0\le j<m$.  We must prove it for $j+1$.

Without loss of generality, we may assume that $P_j$ is
$v_{m-j},v_{m+j}$-path with vertices in $V_{j}$ of minimum length.
Then $P_j$ is an {\em induced path},
that is, there is no edge between non-consecutive vertices of $P_j$.

$P_j$ divides the triangulated disk into two faces; let $f$ be the face that does not
contain $x$.
For any $y_k\in V_k$, there is a path $y_k,y_{k-1},\ldots,y_0$ with each $y_i\in V_i$.
If $y_k$ is in $f$, then this path must intersect $P_j$ at $y_i$ with $0\le i< k$,
which implies that $k> j$.
Therefore, $f$ contains no vertex of $\bigcup_{0\le i\le j}V_i$.

Let $P_j=y_1,y_2,\ldots,y_p$, where $v_{m-j}=y_1$ and $y_p=v_{m+j}$.
A triangle in $f$ with more than one endpoint on $P_j$ must intersect $V(P_j)$
at two consecutive endpoints $y_i,y_{i+1}$,
since $P_j$ is an induced path and $G'$ has no multiple edges.
Every edge $y_iy_{i+1}$ of $P_j$ is incident to exactly one triangle in $f$; let $z_i$
be its third vertex, which is in $f$.  Let $z_0=v_{m-j-1}$ and let $z_p=v_{m+j+1}$.
(See Figure~\ref{hope it works too}.)
For each $y_i$ in $P_j$, there is a set of triangles in $f$ such that $y_i$ is their only vertex in $P_j$, naturally ordered by the embedding near $y_i$; removing $y_i$ yields a $z_{i-1},z_i$-walk in $f$.  Concatenating these walks gives a $z_0,z_p$-walk $W$ in $f$.
Since every vertex of $W$ is incident to a vertex of $P_j$, and $W$ is in $f$, the vertices of $W$ must be contained in $V_{j+1}$.  $W$ contains a $z_0,z_p$-path; let this be $P_{j+1}$.
\qed
\end{proof}

\begin{figure}[ht]
\begin{center}
\definecolor{mygray}{gray}{0.85}
\begin{pspicture}(8,5)
  \psline(-2,0)(10,0)
  \psarc(4,0){4}{0}{180}
  \psarc[fillstyle=solid,fillcolor=mygray](4,0){2}{0}{180}
  \rput(4,0.8){(not $f$)}
  \psdots(4,0)
  \rput(4,-0.4){$v_m$}
  \rput(2,-0.2){$y_1$}
  \rput(2,-0.6){$(v_{m-j})$}
  \rput(6,-0.2){$y_p$}
  \rput(6,-0.6){$(v_{m+j})$}
  \rput(0,-0.23){$z^0$}
  \rput(0,-0.6){$(v_{m-j-1)})$}
  \rput(8,-0.27){$z_p$}
  \rput(8,-0.6){$(v_{m+j+1})$}
  \psline(2,0)(0.2,1.2)
  \rput(-0.15,1.25){$z_1$}
  \psline(0.2,1.2)(2.4,1.2)
  \rput(2.55,1.1){$y_2$}
  \psline(2.4,1.2)(0.8,2.4)
  \rput(0.5,2.55){$z_2$}
  \psline(3.5,1.9)(4,4)
  \psline(4.5,1.9)(4,4)
  \psline(3.5,1.9)(2.9,3.85)
  \psline(4.5,1.9)(5.1,3.85)
  \rput(3.5,1.6){$y_i$}
  \rput(4.5,1.6){$y_{i+1}$}
  \rput(2.8,4.1){$z_{i-1}$}
  \rput(5.3,4.1){$z_{i+1}$}
  \rput(4.1,4.28){$z_i$}
  \rput(6,1){$P_j$}
  \psarc[linestyle=dotted](4,0){3}{15}{60}
  \psarc[linestyle=dotted](4,0){3}{115}{135}

\end{pspicture}
\end{center}
\caption{Construction of $P_{j+1}$ in $f$, when $P_j$ is assumed to be of minimum length}
\label{hope it works too}
\end{figure}
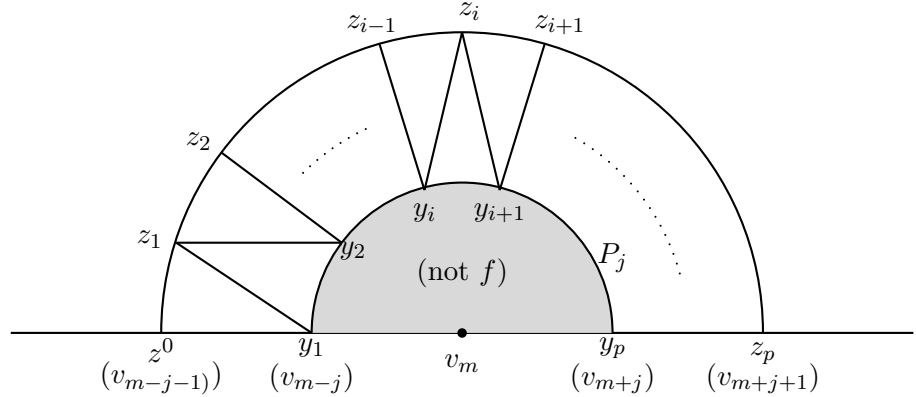

\begin{lemma}\label{proj-plane-noncontractible-cycle}
Any $n$-vertex triangulation on the projective plane has a non-contractible cycle of length less than or equal to $2\sqrt{n}-1$.
\end{lemma}

\begin{proof}
Since the vertex sets $V_0,\ldots,V_m$ are disjoint, the vertex sets of the paths
$P_0,\ldots,P_m$ are disjoint.  Therefore,
$$n\ge \sum_{j=0}^{m}|V(P_j)|\ge\sum_{j=0}^{m}(2j+1)=(m+1)^2.$$
Then $2\sqrt{n}-1\ge 2m+1 = 2\lfloor|C|/2\rfloor+1\ge|C|$.
\qed
\end{proof}

\begin{lemma}\label{noncontractible cycle in nonorientable surface}
Any $n$-vertex triangulation on a non-orientable surface with genus $g>1$
has a non-contractible cycle of length less than or equal to $2\sqrt{n}$.
\end{lemma}

\begin{proof}
For any non-orientable surface $N_g$, it is known that the double cover of $N_g$ is $S_{g-1}$.
Let $G$ be an $n$-vertex triangulation of $N_g$ with $g>1$, and let $H$ be the double cover of $G$.
Then $H$ is a $2n$-vertex triangulation on $S_{g-1}$, an orientable surface that is not the sphere.
By~\cite{JoanHutchinson}, $H$ has a non-contractible cycle $C_H$ of length at most $\sqrt {2(2n)}=2\sqrt{n}$.

$C_H$ maps to a closed walk $C_G$ on $N_g$.
If $C_G$ is contractible, then there is a homotopy in $S_{g-1}$
from $C_G$ to a point.
It lifts to a homotopy in $N_g$ from $C_H$ to a point
(\cite[Lemma 54.2]{Munkres}, for example).
Then $C_H$ is contractible in $N_g$, which is a contradiction.
Therefore, $C_G$ is a non-contractible closed walk in $G$ of length at most $|C_H|$.
$C_G$ must contain a non-contractible cycle, which has length at most $|C_G|\le |C_H|\le 2\sqrt{n}$.
\qed
\end{proof}

The previous two lemmas complete the proof of Theorem~\ref{T:non}.

\section*{Acknowledgements}

We would like to thank Erin W. Chambers for pointing us in the direction of double covers and Erika L.C. King for allowing us to use two figures from~\cite{ArxivKingMJP}.
We would also like to thank the referees for their careful work, which led to critical improvements in
the paper.

\section{Appendix: Figures for cases with $w=5,7,11$}\label{Appendix}

\begin{figure}[ht]
\begin{center}
\begin{tabular}{ccccc}
\begin{pspicture}(2.4,3.2)
\psset{unit=.6cm}
\multido{\iA=0+1}{5}{%
    \rput(\iA,5.5){$\iA$}}
\multido{\iA=0+1,\iB=5+-1}{6}{%
    \rput(-0.5,\iB){$\iA$}}
  \psgrid[subgriddiv=-1,griddots=6,gridlabels=0pt,gridwidth=.5pt](4,5)
  \multido{\rA=0+1,\rB=1+1}{4}{%
    \psline[linestyle=dotted](\rA,0)(\rB,1)}
  \multido{\rA=0+1,\rB=1+1}{4}{%
    \psline[linestyle=dotted](\rA,1)(\rB,2)}
  \multido{\rA=0+1,\rB=1+1}{4}{%
    \psline[linestyle=dotted](\rA,2)(\rB,3)}
  \multido{\rA=0+1,\rB=1+1}{4}{%
    \psline[linestyle=dotted](\rA,3)(\rB,4)}
  \multido{\rA=0+1,\rB=1+1}{4}{%
    \psline[linestyle=dotted](\rA,4)(\rB,5)}
  \psdots(1,1)(1,3)(3,0)(4,3)
  \psdots[dotstyle=o](3,5)
  \pspolygon(0,0)(1,0)(2,1)(2,2)(1,2)(0,1)
  \pspolygon(0,2)(1,2)(2,3)(2,4)(1,4)(0,3)
  \psline(2,0)(3,1)(4,1)(4,0)
  \psline(2,5)(2,4)(3,4)(4,5)
  \psline(4,4)(3,3)(3,2)(4,2)
  \psline(0,5)(1,5)
  \psline(0,3)(0,4)
\end{pspicture}
& \qquad\qquad &
\begin{pspicture}(3.6,3.2)
\psset{unit=.6cm}
\multido{\iA=0+1}{7}{%
    \rput(\iA,5.5){$\iA$}}
\multido{\iA=0+1,\iB=5+-1}{6}{%
    \rput(-0.5,\iB){$\iA$}}
  \psgrid[subgriddiv=1,griddots=6,gridlabels=0pt,gridwidth=.5pt](6,5)
  \multido{\rA=0+1,\rB=1+1}{6}{%
    \psline[linestyle=dotted](\rA,0)(\rB,1)}
  \multido{\rA=0+1,\rB=1+1}{6}{%
    \psline[linestyle=dotted](\rA,1)(\rB,2)}
  \multido{\rA=0+1,\rB=1+1}{6}{%
    \psline[linestyle=dotted](\rA,2)(\rB,3)}
  \multido{\rA=0+1,\rB=1+1}{6}{%
    \psline[linestyle=dotted](\rA,3)(\rB,4)}
  \multido{\rA=0+1,\rB=1+1}{6}{%
    \psline[linestyle=dotted](\rA,5)(\rB,4)}
  \psdots(1,4)(2,2)(3,0)(5,3)(6,1)
  \psdots[dotstyle=o](3,5)
  \pspolygon(1,1)(2,1)(3,2)(3,3)(2,3)(1,2)
  \pspolygon(4,2)(5,2)(6,3)(6,4)(5,4)(4,3)
  \pspolygon(0,3)(1,3)(2,4)(1,5)(0,5)
  \psline(2,0)(3,1)(4,1)(4,0)
  \psline(2,5)(3,4)(4,4)(4,5)
  \psline(6,2)(5,1)(5,0)(6,0)
  \psline(5,5)(6,5)
  \psline(0,0)(1,0)
  \psline(0,1)(0,2)
\end{pspicture}
& \qquad\qquad &
\begin{pspicture}(3.0,3.2)
\psset{unit=.6cm}
\multido{\iA=0+1}{6}{%
    \rput(\iA,5.5){$\iA$}}
\multido{\iA=0+1,\iB=5+-1}{6}{%
    \rput(-0.5,\iB){$\iA$}}
  \psgrid[subgriddiv=1,griddots=6,gridlabels=0pt,gridwidth=.5pt](5,5)
  \multido{\rA=0+1,\rB=1+1}{5}{%
    \psline[linestyle=dotted](\rA,0)(\rB,1)}
  \multido{\rA=0+1,\rB=1+1}{5}{%
    \psline[linestyle=dotted](\rA,1)(\rB,2)}
  \multido{\rA=0+1,\rB=1+1}{5}{%
    \psline[linestyle=dotted](\rA,2)(\rB,3)}
  \multido{\rA=0+1,\rB=1+1}{5}{%
    \psline[linestyle=dotted](\rA,4)(\rB,3)}
  \multido{\rA=0+1,\rB=1+1}{5}{%
    \psline[linestyle=dotted](\rA,5)(\rB,4)}
  \psdots(0,3)(2,4)(1,1)(4,0)(4,2)
  \psdots[dotstyle=o](4,5)
  \pspolygon(0,0)(1,0)(2,1)(2,2)(1,2)(0,1)
  \pspolygon(3,1)(4,1)(5,2)(5,3)(4,3)(3,2)
  \pspolygon(2,3)(3,3)(3,4)(2,5)(1,5)(1,4)
  \psline(0,2)(1,3)(0,4)
  \psline(3,0)(4,1)(5,1)(5,0)
  \psline(3,5)(4,4)(5,4)(5,5)
  \psline(1,0)(2,0)
  \psline(0,5)(1,5)
\end{pspicture}
\end{tabular}
\end{center}
\caption{Dominating sets for $w=5$, with $k=0$ (left), $k=1$ (center), and $k=2$ (right)}
\label{F:w=5 k=012}
\end{figure}
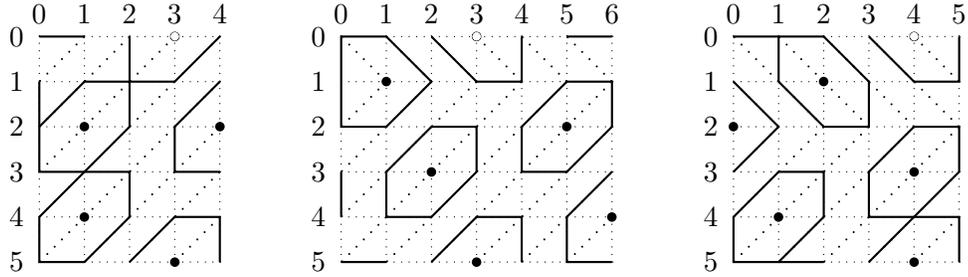

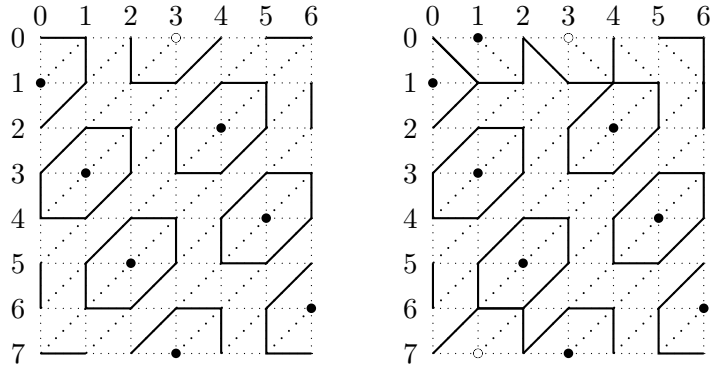
\begin{figure}[ht]
\begin{center}
\begin{tabular}{ccc}
\begin{pspicture}(3.6,4.4)
\psset{unit=.6cm}
\multido{\iA=0+1}{7}{%
    \rput(\iA,7.5){$\iA$}}
\multido{\iA=0+1,\iB=7+-1}{8}{%
    \rput(-0.5,\iB){$\iA$}}
  \psgrid[subgriddiv=1,griddots=6,gridlabels=0pt,gridwidth=.5pt](6,7)
  \multido{\rA=0+1,\rB=1+1}{6}{%
    \psline[linestyle=dotted](\rA,0)(\rB,1)}
  \multido{\rA=0+1,\rB=1+1}{6}{%
    \psline[linestyle=dotted](\rA,1)(\rB,2)}
  \multido{\rA=0+1,\rB=1+1}{6}{%
    \psline[linestyle=dotted](\rA,2)(\rB,3)}
  \multido{\rA=0+1,\rB=1+1}{6}{%
    \psline[linestyle=dotted](\rA,3)(\rB,4)}
  \multido{\rA=0+1,\rB=1+1}{6}{%
    \psline[linestyle=dotted](\rA,4)(\rB,5)}
  \multido{\rA=0+1,\rB=1+1}{6}{%
    \psline[linestyle=dotted](\rA,5)(\rB,6)}
  \multido{\rA=0+1,\rB=1+1}{6}{%
    \psline[linestyle=dotted](\rA,6)(\rB,7)}
  \psdots(0,6)(1,4)(2,2)(3,0)(4,5)(5,3)(6,1)
  \psdots[dotstyle=o](3,7)
  \pspolygon(1,1)(2,1)(3,2)(3,3)(2,3)(1,2)
  \pspolygon(0,3)(1,3)(2,4)(2,5)(1,5)(0,4)
  \pspolygon(3,4)(4,4)(5,5)(5,6)(4,6)(3,5)
  \pspolygon(4,2)(5,2)(6,3)(6,4)(5,4)(4,3)
  \psline(0,5)(1,6)(1,7)(0,7)
  \psline(2,7)(2,6)(3,6)(4,7)
  \psline(2,0)(3,1)(4,1)(4,0)
  \psline(6,0)(5,0)(5,1)(6,2)
  \psline(0,0)(1,0)
  \psline(0,1)(0,2)
  \psline(5,7)(6,7)
  \psline(6,6)(6,5)
\end{pspicture}
& \qquad\qquad &
\begin{pspicture}(3.6,4.4)
\psset{unit=.6cm}
\multido{\iA=0+1}{7}{%
    \rput(\iA,7.5){$\iA$}}
\multido{\iA=0+1,\iB=7+-1}{8}{%
    \rput(-0.5,\iB){$\iA$}}
  \psgrid[subgriddiv=1,griddots=6,gridlabels=0pt,gridwidth=.5pt](6,7)
  \multido{\rA=0+1,\rB=1+1}{6}{%
    \psline[linestyle=dotted](\rA,0)(\rB,1)}
  \multido{\rA=0+1,\rB=1+1}{6}{%
    \psline[linestyle=dotted](\rA,1)(\rB,2)}
  \multido{\rA=0+1,\rB=1+1}{6}{%
    \psline[linestyle=dotted](\rA,2)(\rB,3)}
  \multido{\rA=0+1,\rB=1+1}{6}{%
    \psline[linestyle=dotted](\rA,3)(\rB,4)}
  \multido{\rA=0+1,\rB=1+1}{6}{%
    \psline[linestyle=dotted](\rA,4)(\rB,5)}
  \multido{\rA=0+1,\rB=1+1}{6}{%
    \psline[linestyle=dotted](\rA,5)(\rB,6)}
  \multido{\rA=0+1,\rB=1+1}{6}{%
    \psline[linestyle=dotted](\rA,7)(\rB,6)}
  \psdots(0,6)(1,4)(2,2)(3,0)(4,5)(5,3)(6,1)(1,7)
  \psdots[dotstyle=o](1,0)(3,7)
  \pspolygon(1,1)(2,1)(3,2)(3,3)(2,3)(1,2)
  \pspolygon(0,3)(1,3)(2,4)(2,5)(1,5)(0,4)
  \pspolygon(3,4)(4,4)(5,5)(5,6)(4,6)(3,5)
  \pspolygon(4,2)(5,2)(6,3)(6,4)(5,4)(4,3)
  \psline(0,5)(1,6)(0,7)
  \psline(2,7)(3,6)(4,6)(4,7)
  \psline(0,7)(1,6)(2,6)(2,7)
  \psline(2,0)(3,1)(4,1)(4,0)
  \psline(0,0)(1,1)(2,1)(2,0)
  \psline(6,0)(5,0)(5,1)(6,2)
  \psline(6,7)(6,5)
  \psline(0,1)(0,2)
  \psline(5,7)(6,7)
  \psline(6,6)(6,5)
\end{pspicture}
\end{tabular}
\end{center}
\caption{Dominating sets for $w=7$ with $k=0$ (left) and $k=1$ (right)}
\label{F:w=7 k=01}
\end{figure}

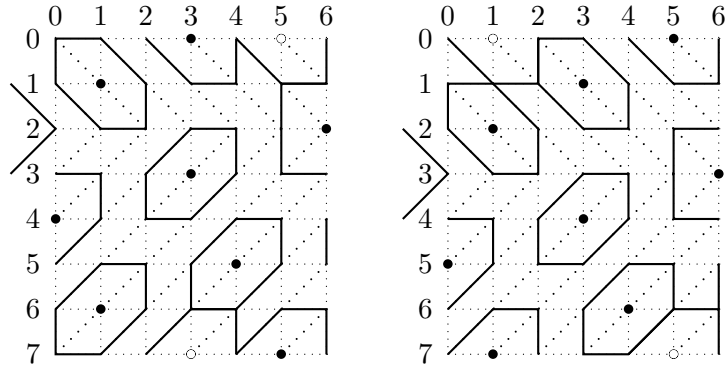
\begin{figure}[ht]
\begin{center}
\begin{tabular}{ccc}
\begin{pspicture}(3.6,4.4)
\psset{unit=.6cm}
\multido{\iA=0+1}{7}{%
    \rput(\iA,7.5){$\iA$}}
\multido{\iA=0+1,\iB=7+-1}{8}{%
    \rput(-0.5,\iB){$\iA$}}
  \psgrid[subgriddiv=1,griddots=6,gridlabels=0pt,gridwidth=.5pt](6,7)
  \multido{\rA=0+1,\rB=1+1}{6}{%
    \psline[linestyle=dotted](\rA,0)(\rB,1)}
  \multido{\rA=0+1,\rB=1+1}{6}{%
    \psline[linestyle=dotted](\rA,1)(\rB,2)}
  \multido{\rA=0+1,\rB=1+1}{6}{%
    \psline[linestyle=dotted](\rA,2)(\rB,3)}
  \multido{\rA=0+1,\rB=1+1}{6}{%
    \psline[linestyle=dotted](\rA,3)(\rB,4)}
  \multido{\rA=0+1,\rB=1+1}{6}{%
    \psline[linestyle=dotted](\rA,4)(\rB,5)}
  \multido{\rA=0+1,\rB=1+1}{6}{%
    \psline[linestyle=dotted](\rA,6)(\rB,5)}
  \multido{\rA=0+1,\rB=1+1}{6}{%
    \psline[linestyle=dotted](\rA,7)(\rB,6)}
  \psdots(0,3)(1,1)(1,6)(3,4)(3,7)(4,2)(5,0)(6,5)
  \psdots[dotstyle=o](3,0)(5,7)
  \pspolygon(0,0)(1,0)(2,1)(2,2)(1,2)(0,1)
  \pspolygon(0,7)(0,6)(1,5)(2,5)(2,6)(1,7)
  \pspolygon(2,3)(3,3)(4,4)(4,5)(3,5)(2,4)
  \pspolygon(3,1)(4,1)(5,2)(5,3)(4,3)(3,2)
  \psline(6,6)(5,6)(5,4)(6,4)
  \psline(-1,6)(0,5)(-1,4)
  \psline(0,4)(1,4)(1,3)(0,2)
  \psline(6,2)(6,3)
  \psline(4,0)(5,1)(6,1)(6,0)
  \psline(4,7)(5,6)(6,6)(6,7)
  \psline(2,0)(3,1)(4,1)(4,0)
  \psline(2,7)(3,6)(4,6)(4,7)
\end{pspicture}
& \qquad\qquad &
\begin{pspicture}(3.6,4.4)
\psset{unit=.6cm}
\multido{\iA=0+1}{7}{%
    \rput(\iA,7.5){$\iA$}}
\multido{\iA=0+1,\iB=7+-1}{8}{%
    \rput(-0.5,\iB){$\iA$}}
  \psgrid[subgriddiv=1,griddots=6,gridlabels=0pt,gridwidth=.5pt](6,7)
  \multido{\rA=0+1,\rB=1+1}{6}{%
    \psline[linestyle=dotted](\rA,0)(\rB,1)}
  \multido{\rA=0+1,\rB=1+1}{6}{%
    \psline[linestyle=dotted](\rA,1)(\rB,2)}
  \multido{\rA=0+1,\rB=1+1}{6}{%
    \psline[linestyle=dotted](\rA,2)(\rB,3)}
  \multido{\rA=0+1,\rB=1+1}{6}{%
    \psline[linestyle=dotted](\rA,3)(\rB,4)}
  \multido{\rA=0+1,\rB=1+1}{6}{%
    \psline[linestyle=dotted](\rA,5)(\rB,4)}
  \multido{\rA=0+1,\rB=1+1}{6}{%
    \psline[linestyle=dotted](\rA,6)(\rB,5)}
  \multido{\rA=0+1,\rB=1+1}{6}{%
    \psline[linestyle=dotted](\rA,7)(\rB,6)}
  \psdots(0,2)(1,0)(1,5)(3,3)(3,6)(4,1)(5,7)(6,4)
  \psdots[dotstyle=o](5,0)(1,7)
  \pspolygon(3,5)(4,5)(4,6)(3,7)(2,7)(2,6)
  \pspolygon(0,6)(0,5)(1,4)(2,4)(2,5)(1,6)
  \pspolygon(2,2)(3,2)(4,3)(4,4)(3,4)(2,3)
  \pspolygon(3,0)(4,0)(5,1)(5,2)(4,2)(3,1)
  \psline(6,5)(5,5)(5,3)(6,3)
  \psline(-1,5)(0,4)(-1,3)
  \psline(0,0)(1,1)(2,1)(2,0)
  \psline(0,7)(1,6)(2,6)(2,7)
  \psline(0,1)(1,2)(1,3)(0,3)
  \psline(6,1)(6,2)
  \psline(4,7)(5,6)(6,6)(6,7)
  \psline(4,0)(5,1)(6,1)(6,0)
\end{pspicture}
\end{tabular}
\end{center}
\caption{Dominating sets for $w=7$ with $k=2$ (left) and $k=3$ (right)}
\label{F:w=7 k=23}
\end{figure}

\begin{figure}[ht]
\begin{center}
\begin{tabular}{ccccc}
\begin{pspicture}(3.6,6.8)
\psset{unit=.6cm}
\multido{\iA=0+1}{7}{%
    \rput(\iA,11.5){$\iA$}}
\multido{\iA=0+1,\iB=11+-1}{12}{%
    \rput(-0.5,\iB){$\iA$}}
  \psgrid[subgriddiv=1,griddots=6,gridlabels=0pt,gridwidth=.5pt](6,11)
  \multido{\rA=0+1,\rB=1+1}{6}{%
    \psline[linestyle=dotted](\rA,0)(\rB,1)}
  \multido{\rA=0+1,\rB=1+1}{6}{%
    \psline[linestyle=dotted](\rA,1)(\rB,2)}
  \multido{\rA=0+1,\rB=1+1}{6}{%
    \psline[linestyle=dotted](\rA,2)(\rB,3)}
  \multido{\rA=0+1,\rB=1+1}{6}{%
    \psline[linestyle=dotted](\rA,3)(\rB,4)}
  \multido{\rA=0+1,\rB=1+1}{6}{%
    \psline[linestyle=dotted](\rA,4)(\rB,5)}
  \multido{\rA=0+1,\rB=1+1}{6}{%
    \psline[linestyle=dotted](\rA,5)(\rB,6)}
  \multido{\rA=0+1,\rB=1+1}{6}{%
    \psline[linestyle=dotted](\rA,6)(\rB,7)}
  \multido{\rA=0+1,\rB=1+1}{6}{%
    \psline[linestyle=dotted](\rA,7)(\rB,8)}
  \multido{\rA=0+1,\rB=1+1}{6}{%
    \psline[linestyle=dotted](\rA,8)(\rB,9)}
  \multido{\rA=0+1,\rB=1+1}{6}{%
    \psline[linestyle=dotted](\rA,9)(\rB,10)}
  \multido{\rA=0+1,\rB=1+1}{6}{%
    \psline[linestyle=dotted](\rA,10)(\rB,11)}
  \psdots(1,0)(0,6)(1,4)(2,2)(3,0)(2,9)(3,7)(4,5)(5,3)(6,1)(5,10)(6,8)
  \psdots[dotstyle=o](1,11)(3,11)
  \pspolygon(1,8)(2,8)(3,9)(3,10)(2,10)(1,9)
  \pspolygon(2,6)(3,6)(4,7)(4,8)(3,8)(2,7)
  \pspolygon(3,4)(4,4)(5,5)(5,6)(4,6)(3,5)
  \pspolygon(4,2)(5,2)(6,3)(6,4)(5,4)(4,3)
  \pspolygon(1,1)(2,1)(3,2)(3,3)(2,3)(1,2)
  \pspolygon(0,3)(1,3)(2,4)(2,5)(1,5)(0,4)
  \pspolygon(4,9)(5,9)(6,10)(6,11)(5,11)(4,10)
  \psline(2,0)(3,1)(4,1)(4,0)
  \psline(2,11)(2,10)(3,10)(4,11)
  \psline(0,0)(1,1)(2,1)(2,0)
  \psline(0,11)(0,10)(1,10)(2,11)
  \psline(6,6)(6,5)
  \psline(6,9)(5,8)(5,7)(6,7)
  \psline(0,8)(0,9)
  \psline(6,2)(5,1)(5,0)(6,0)
  \psline(0,1)(0,2)
  \psline(0,7)(1,7)(1,6)(0,5)
\end{pspicture}
& \qquad\qquad &
\begin{pspicture}(3.6,6.8)
\psset{unit=.6cm}
\multido{\iA=0+1}{7}{%
    \rput(\iA,11.5){$\iA$}}
\multido{\iA=0+1,\iB=11+-1}{12}{%
    \rput(-0.5,\iB){$\iA$}}
  \psgrid[subgriddiv=1,griddots=6,gridlabels=0pt,gridwidth=.5pt](6,11)
  \multido{\rA=0+1,\rB=1+1}{6}{%
    \psline[linestyle=dotted](\rA,0)(\rB,1)}
  \multido{\rA=0+1,\rB=1+1}{6}{%
    \psline[linestyle=dotted](\rA,1)(\rB,2)}
  \multido{\rA=0+1,\rB=1+1}{6}{%
    \psline[linestyle=dotted](\rA,2)(\rB,3)}
  \multido{\rA=0+1,\rB=1+1}{6}{%
    \psline[linestyle=dotted](\rA,3)(\rB,4)}
  \multido{\rA=0+1,\rB=1+1}{6}{%
    \psline[linestyle=dotted](\rA,4)(\rB,5)}
  \multido{\rA=0+1,\rB=1+1}{6}{%
    \psline[linestyle=dotted](\rA,5)(\rB,6)}
  \multido{\rA=0+1,\rB=1+1}{6}{%
    \psline[linestyle=dotted](\rA,6)(\rB,7)}
  \multido{\rA=0+1,\rB=1+1}{6}{%
    \psline[linestyle=dotted](\rA,7)(\rB,8)}
  \multido{\rA=0+1,\rB=1+1}{6}{%
    \psline[linestyle=dotted](\rA,8)(\rB,9)}
  \multido{\rA=0+1,\rB=1+1}{6}{%
    \psline[linestyle=dotted](\rA,9)(\rB,10)}
  \multido{\rA=0+1,\rB=1+1}{6}{%
    \psline[linestyle=dotted](\rA,11)(\rB,10)}
  \psdots(1,10)(0,6)(1,4)(2,2)(3,0)(2,9)(3,7)(4,5)(5,3)(6,1)(5,10)(6,8)
  \psdots[dotstyle=o](3,11)
  \pspolygon(1,8)(2,8)(3,9)(3,10)(2,10)(1,9)
  \pspolygon(2,6)(3,6)(4,7)(4,8)(3,8)(2,7)
  \pspolygon(3,4)(4,4)(5,5)(5,6)(4,6)(3,5)
  \pspolygon(4,2)(5,2)(6,3)(6,4)(5,4)(4,3)
  \pspolygon(1,1)(2,1)(3,2)(3,3)(2,3)(1,2)
  \pspolygon(0,3)(1,3)(2,4)(2,5)(1,5)(0,4)
  \pspolygon(4,9)(5,9)(6,10)(5,11)(4,11)
  \pspolygon(0,9)(1,9)(2,10)(1,11)(0,11)
  \psline(0,0)(1,0)
  \psline(2,0)(3,1)(4,1)(4,0)
  \psline(2,11)(2,10)(3,10)(4,11)
  \psline(6,2)(5,1)(5,0)(6,0)
  \psline(0,1)(0,2)
  \psline(6,6)(6,5)
  \psline(6,9)(5,8)(5,7)(6,7)
  \psline(0,8)(0,9)
  \psline(5,11)(6,11)
  \psline(0,7)(1,7)(1,6)(0,5)
\end{pspicture}
& \qquad\qquad &
\begin{pspicture}(2.4,6.8)
\psset{unit=.6cm}
\multido{\iA=0+1}{5}{%
    \rput(\iA,11.5){$\iA$}}
\multido{\iA=0+1,\iB=11+-1}{12}{%
    \rput(-0.5,\iB){$\iA$}}
  \psgrid[subgriddiv=1,griddots=6,gridlabels=0pt,gridwidth=.5pt](4,11)
  \multido{\rA=0+1,\rB=1+1}{4}{%
    \psline[linestyle=dotted](\rA,0)(\rB,1)}
  \multido{\rA=0+1,\rB=1+1}{4}{%
    \psline[linestyle=dotted](\rA,1)(\rB,2)}
  \multido{\rA=0+1,\rB=1+1}{4}{%
    \psline[linestyle=dotted](\rA,2)(\rB,3)}
  \multido{\rA=0+1,\rB=1+1}{4}{%
    \psline[linestyle=dotted](\rA,3)(\rB,4)}
  \multido{\rA=0+1,\rB=1+1}{4}{%
    \psline[linestyle=dotted](\rA,4)(\rB,5)}
  \multido{\rA=0+1,\rB=1+1}{4}{%
    \psline[linestyle=dotted](\rA,5)(\rB,6)}
  \multido{\rA=0+1,\rB=1+1}{4}{%
    \psline[linestyle=dotted](\rA,6)(\rB,7)}
  \multido{\rA=0+1,\rB=1+1}{4}{%
    \psline[linestyle=dotted](\rA,7)(\rB,8)}
  \multido{\rA=0+1,\rB=1+1}{4}{%
    \psline[linestyle=dotted](\rA,8)(\rB,9)}
  \multido{\rA=0+1,\rB=1+1}{4}{%
    \psline[linestyle=dotted](\rA,10)(\rB,9)}
  \multido{\rA=0+1,\rB=1+1}{4}{%
    \psline[linestyle=dotted](\rA,11)(\rB,10)}
  \psdots(1,10)(1,7)(1,4)(2,2)(3,6)(3,10)(4,1)(4,4)(4,8)
  \pspolygon(1,1)(2,1)(3,2)(3,3)(2,3)(1,2)
  \pspolygon(0,3)(1,3)(2,4)(2,5)(1,5)(0,4)
  \pspolygon(0,6)(1,6)(2,7)(2,8)(1,8)(0,7)
  \pspolygon(2,5)(3,5)(4,6)(4,7)(3,7)(2,6)
  \pspolygon(1,9)(2,9)(2,10)(1,11)(0,11)(0,10)
  \pspolygon(3,9)(4,9)(4,10)(3,11)(2,11)(2,10)
  \psline(4,0)(3,0)(3,1)(4,2)
  \psline(4,3)(3,3)(3,4)(4,5)
  \psline(4,7)(3,7)(3,8)(4,9)
  \psline(0,1)(0,2)
  \psline(0,4)(0,5)
  \psline(0,8)(0,9)
  \psline(0,0)(1,0)
  \psline(2,0)(3,0)
  \psline(3,11)(4,11)
\end{pspicture}
\end{tabular}
\end{center}
\caption{Dominating sets for $w=11$ with $k=0$ (left), $k=1$ (center), and $k=2$ (right)}
\label{F:w=11 k=012}
\end{figure}
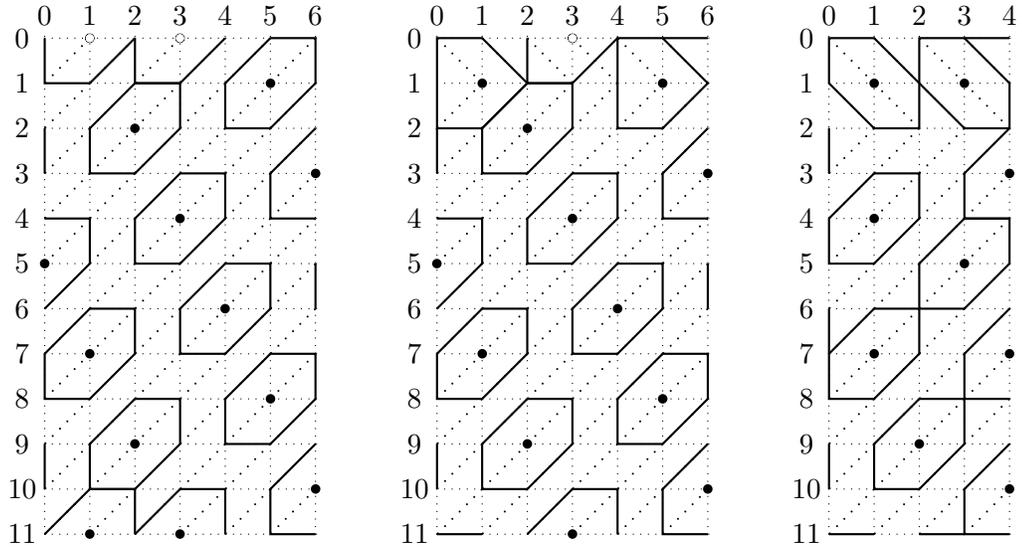

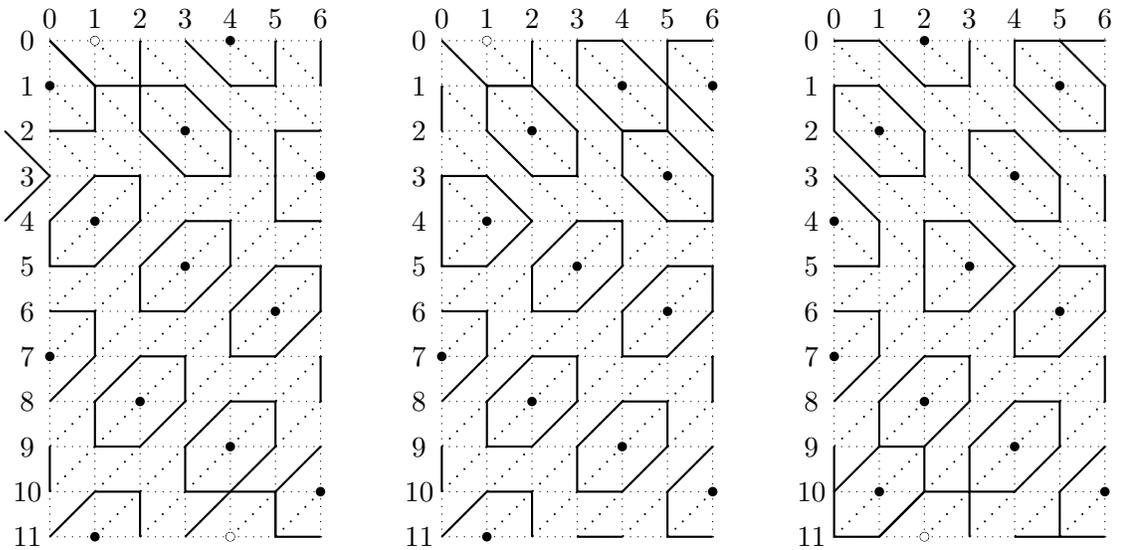
\begin{figure}[ht]
\begin{center}
\begin{tabular}{ccccc}
\begin{pspicture}(3.6,6.8)
\psset{unit=.6cm}
\multido{\iA=0+1}{7}{%
    \rput(\iA,11.5){$\iA$}}
\multido{\iA=0+1,\iB=11+-1}{12}{%
    \rput(-0.5,\iB){$\iA$}}
  \psgrid[subgriddiv=1,griddots=6,gridlabels=0pt,gridwidth=.5pt](6,11)
  \multido{\rA=0+1,\rB=1+1}{6}{%
    \psline[linestyle=dotted](\rA,0)(\rB,1)}
  \multido{\rA=0+1,\rB=1+1}{6}{%
    \psline[linestyle=dotted](\rA,1)(\rB,2)}
  \multido{\rA=0+1,\rB=1+1}{6}{%
    \psline[linestyle=dotted](\rA,2)(\rB,3)}
  \multido{\rA=0+1,\rB=1+1}{6}{%
    \psline[linestyle=dotted](\rA,3)(\rB,4)}
  \multido{\rA=0+1,\rB=1+1}{6}{%
    \psline[linestyle=dotted](\rA,4)(\rB,5)}
  \multido{\rA=0+1,\rB=1+1}{6}{%
    \psline[linestyle=dotted](\rA,5)(\rB,6)}
  \multido{\rA=0+1,\rB=1+1}{6}{%
    \psline[linestyle=dotted](\rA,6)(\rB,7)}
  \multido{\rA=0+1,\rB=1+1}{6}{%
    \psline[linestyle=dotted](\rA,7)(\rB,8)}
  \multido{\rA=0+1,\rB=1+1}{6}{%
    \psline[linestyle=dotted](\rA,9)(\rB,8)}
  \multido{\rA=0+1,\rB=1+1}{6}{%
    \psline[linestyle=dotted](\rA,10)(\rB,9)}
  \multido{\rA=0+1,\rB=1+1}{6}{%
    \psline[linestyle=dotted](\rA,11)(\rB,10)}
  \psdots(1,0)(0,4)(2,3)(4,2)(6,1)(1,7)(3,6)(5,5)(0,10)(3,9)(6,8)(4,11)
  \psdots[dotstyle=o](4,0)(1,11)
  \pspolygon(1,2)(2,2)(3,3)(3,4)(2,4)(1,3)
  \pspolygon(3,1)(4,1)(5,2)(5,3)(4,3)(3,2)
  \pspolygon(4,4)(5,4)(6,5)(6,6)(5,6)(4,5)
  \pspolygon(2,5)(3,5)(4,6)(4,7)(3,7)(2,6)
  \psline(0,0)(1,1)(2,1)(2,0)
  \psline(0,11)(1,10)(2,10)(2,11)
  \psline(6,2)(5,1)(5,0)(6,0)
  \psline(0,1)(0,2)
  \psline(0,5)(1,5)(1,4)(0,3)
  \psline(6,3)(6,4)
  \pspolygon(0,6)(1,6)(2,7)(2,8)(1,8)(0,7)
  \pspolygon(2,9)(3,8)(4,8)(4,9)(3,10)(2,10)
  \psline(-1,7)(0,8)(-1,9)
  \psline(6,9)(5,9)(5,7)(6,7)
  \psline(0,9)(1,9)(1,10)(0,11)
  \psline(6,10)(6,11)
  \psline(3,0)(4,1)(5,1)(5,0)
  \psline(3,11)(4,10)(5,10)(5,11)
\end{pspicture}
& \qquad\qquad &
\begin{pspicture}(3.6,6.8)
\psset{unit=.6cm}
\multido{\iA=0+1}{7}{%
    \rput(\iA,11.5){$\iA$}}
\multido{\iA=0+1,\iB=11+-1}{12}{%
    \rput(-0.5,\iB){$\iA$}}
  \psgrid[subgriddiv=1,griddots=6,gridlabels=0pt,gridwidth=.5pt](6,11)
  \multido{\rA=0+1,\rB=1+1}{6}{%
    \psline[linestyle=dotted](\rA,0)(\rB,1)}
  \multido{\rA=0+1,\rB=1+1}{6}{%
    \psline[linestyle=dotted](\rA,1)(\rB,2)}
  \multido{\rA=0+1,\rB=1+1}{6}{%
    \psline[linestyle=dotted](\rA,2)(\rB,3)}
  \multido{\rA=0+1,\rB=1+1}{6}{%
    \psline[linestyle=dotted](\rA,3)(\rB,4)}
  \multido{\rA=0+1,\rB=1+1}{6}{%
    \psline[linestyle=dotted](\rA,4)(\rB,5)}
  \multido{\rA=0+1,\rB=1+1}{6}{%
    \psline[linestyle=dotted](\rA,5)(\rB,6)}
  \multido{\rA=0+1,\rB=1+1}{6}{%
    \psline[linestyle=dotted](\rA,6)(\rB,7)}
  \multido{\rA=0+1,\rB=1+1}{6}{%
    \psline[linestyle=dotted](\rA,8)(\rB,7)}
  \multido{\rA=0+1,\rB=1+1}{6}{%
    \psline[linestyle=dotted](\rA,9)(\rB,8)}
  \multido{\rA=0+1,\rB=1+1}{6}{%
    \psline[linestyle=dotted](\rA,10)(\rB,9)}
  \multido{\rA=0+1,\rB=1+1}{6}{%
    \psline[linestyle=dotted](\rA,11)(\rB,10)}
  \psdots(1,0)(0,4)(2,3)(4,2)(6,1)(1,7)(3,6)(5,5)(2,9)(5,8)(6,10)(4,10)
  \psdots[dotstyle=o](1,11)
  \pspolygon(1,2)(2,2)(3,3)(3,4)(2,4)(1,3)
  \pspolygon(3,1)(4,1)(5,2)(5,3)(4,3)(3,2)
  \pspolygon(4,4)(5,4)(6,5)(6,6)(5,6)(4,5)
  \pspolygon(2,5)(3,5)(4,6)(4,7)(3,7)(2,6)
  \psline(0,0)(1,1)(2,1)(2,0)
  \psline(0,11)(1,10)(2,10)(2,11)
  \psline(6,2)(5,1)(5,0)(6,0)
  \psline(0,1)(0,2)
  \psline(0,5)(1,5)(1,4)(0,3)
  \psline(6,3)(6,4)
  \pspolygon(0,6)(1,6)(2,7)(1,8)(0,8)
  \pspolygon(1,9)(2,8)(3,8)(3,9)(2,10)(1,10)
  \pspolygon(4,8)(5,7)(6,7)(6,8)(5,9)(4,9)
  \pspolygon(3,10)(4,9)(5,9)(5,10)(4,11)(3,11)
  \psline(0,9)(0,10)
  \psline(6,11)(5,11)(5,10)(6,9)
  \psline(3,0)(4,0)
\end{pspicture}
& \qquad\qquad &
\begin{pspicture}(3.6,6.8)
\psset{unit=.6cm}
\multido{\iA=0+1}{7}{%
    \rput(\iA,11.5){$\iA$}}
\multido{\iA=0+1,\iB=11+-1}{12}{%
    \rput(-0.5,\iB){$\iA$}}
  \psgrid[subgriddiv=1,griddots=6,gridlabels=0pt,gridwidth=.5pt](6,11)
  \multido{\rA=0+1,\rB=1+1}{6}{%
    \psline[linestyle=dotted](\rA,0)(\rB,1)}
  \multido{\rA=0+1,\rB=1+1}{6}{%
    \psline[linestyle=dotted](\rA,1)(\rB,2)}
  \multido{\rA=0+1,\rB=1+1}{6}{%
    \psline[linestyle=dotted](\rA,2)(\rB,3)}
  \multido{\rA=0+1,\rB=1+1}{6}{%
    \psline[linestyle=dotted](\rA,3)(\rB,4)}
  \multido{\rA=0+1,\rB=1+1}{6}{%
    \psline[linestyle=dotted](\rA,4)(\rB,5)}
  \multido{\rA=0+1,\rB=1+1}{6}{%
    \psline[linestyle=dotted](\rA,5)(\rB,6)}
  \multido{\rA=0+1,\rB=1+1}{6}{%
    \psline[linestyle=dotted](\rA,7)(\rB,6)}
  \multido{\rA=0+1,\rB=1+1}{6}{%
    \psline[linestyle=dotted](\rA,8)(\rB,7)}
  \multido{\rA=0+1,\rB=1+1}{6}{%
    \psline[linestyle=dotted](\rA,9)(\rB,8)}
  \multido{\rA=0+1,\rB=1+1}{6}{%
    \psline[linestyle=dotted](\rA,10)(\rB,9)}
  \multido{\rA=0+1,\rB=1+1}{6}{%
    \psline[linestyle=dotted](\rA,11)(\rB,10)}
  \psdots(0,4)(2,3)(4,2)(6,1)(0,7)(3,6)(5,5)(1,9)(4,8)(5,10)(2,11)(1,1)
  \psdots[dotstyle=o](2,0)
  \pspolygon(1,2)(2,2)(3,3)(3,4)(2,4)(1,3)
  \pspolygon(3,1)(4,1)(5,2)(5,3)(4,3)(3,2)
  \pspolygon(4,4)(5,4)(6,5)(6,6)(5,6)(4,5)
  \pspolygon(2,5)(3,5)(4,6)(3,7)(2,7)
  \psline(1,0)(2,1)(3,1)(3,0)
  \psline(1,11)(2,10)(3,10)(3,11)
  \psline(6,2)(5,1)(5,0)(6,0)
  \psline(0,1)(0,2)
  \psline(0,5)(1,5)(1,4)(0,3)
  \psline(6,3)(6,4)
  \psline(4,0)(5,0)
  \psline(5,11)(6,11)
  \psline(0,6)(1,6)(1,7)(0,8)
  \psline(6,7)(6,8)
  \pspolygon(0,9)(1,8)(2,8)(2,9)(1,10)(0,10)
  \pspolygon(3,8)(4,7)(5,7)(5,8)(4,9)(3,9)
  \pspolygon(4,10)(5,9)(6,9)(6,10)(5,11)(4,11)
  \psline(0,11)(1,11)
  \pspolygon(0,0)(1,0)(2,1)(2,2)(1,2)(0,1)
\end{pspicture}
\end{tabular}
\end{center}
\caption{Dominating sets for $w=11$ with $k=3$ (left), $k=4$ (center), and $k=5$ (right)}
\label{F:w=11 k=345}
\end{figure}

\end{document}